\documentclass[a4paper,reqno,11pt]{amsart}

\usepackage[utf8]{inputenc}
\usepackage{microtype}

\usepackage{amssymb}
\usepackage{mathrsfs}
\usepackage{mathtools}
\usepackage{comment}

\usepackage{color}
\usepackage{tikz}
\usetikzlibrary{calc}
\usetikzlibrary{topaths}

\usepackage{hyperref}

\newtheorem{theorem}{Theorem}[section]
\newtheorem*{theorem*}{Theorem}
\newtheorem*{maintheorem}{Main Theorem}
\newtheorem{lemma}[theorem]{Lemma}
\newtheorem*{morselemma}{Morse Lemma}
\newtheorem{proposition}[theorem]{Proposition}
\newtheorem*{proposition*}{Proposition}
\newtheorem*{brownscriterion}{Brown's Criterion}
\newtheorem{corollary}[theorem]{Corollary}
\newtheorem*{corollary*}{Corollary}

\newtheorem*{conjecture*}{Conjecture}

\theoremstyle{definition}
\newtheorem{definition}[theorem]{Definition}
\newtheorem{remark}[theorem]{Remark}
\newtheorem{observation}[theorem]{Observation}

\newcommand{\Z}{\mathbb{Z}}
\newcommand{\N}{\mathbb{N}}
\newcommand{\R}{\mathbb{R}}
\newcommand{\Poset}{\mathcal{P}}
\newcommand{\spraige}{\mathscr{S}}
\newcommand{\match}{\mathcal{M}}

\newcommand{\matcharc}{\mathcal{MA}}

\DeclareMathOperator{\Stab}{Stab}
\DeclareMathOperator{\lk}{lk}
\DeclareMathOperator{\dlk}{{\lk}{\downarrow}}

\newcommand{\floor}[1]{\lfloor #1 \rfloor}

\newcommand{\elbraigecpx}{\mathcal{EB}}
\newcommand{\velbraigecpx}{\mathcal{VEB}}

\DeclareMathOperator*{\Bigjoin}{\text{\raisebox{-0.65ex}{\Huge $\ast$}}}

\newcommand{\defeq}{\mathrel{\vcentcolon =}}

\DeclareMathOperator{\id}{id}
\DeclareMathOperator{\F}{F}

\newcommand{\Vbr}%
   {V_{\operatorname{br}}}                 

\numberwithin{equation}{section}

\newlength{\stdbskip}
\newlength{\bskipabstract}
\newlength{\bskipmaintext}
\newlength{\deltatextwidth}
\newlength{\deltatextheight}

\begin{document}

\title{The braided Brin-Thompson groups}
\date{\today}


\author[R. Spahn]{Robert Spahn}
\address{Department of Mathematics and Statistics, University at Albany (SUNY), Albany, NY 12222}
\email{rspahn@albany.edu}

\begin{abstract}
\setlength{\baselineskip}{\bskipabstract}
We construct braided versions $s\Vbr$ of the Brin-Thompson groups $sV$ and prove that they are of type $\F_\infty$. The proof involves showing that the matching complexes of colored arcs on surfaces are highly connected.
\end{abstract}

\maketitle
\thispagestyle{empty}


\setlength{\baselineskip}{\bskipmaintext}

\noindent
\section*{Introduction}
A group is of \emph{type} $F_\infty$ if it admits a classifying space whose $n$-skeleton is compact for every $n$. Some examples
of groups of type $\F_\infty$ include Thompson's group $V$, the Brin-Thompson group $sV$ for $s \in \N$, and the braid group on $n$ strands $B_n$.  A braided version of Thompson's group $V$, which we will denote $\Vbr$, was introduced independently by Brin and Dehornoy in \cite{brin07,dehornoy06}. This group contains copies of the braid group $B_n$ for each $n\in\N$, and was shown to be finitely presented by Brin in \cite{brin06} and to be of type $\F_\infty$ by Bux-Fluch-Marschler-Witzel-Zaremsky in \cite{Bux16}. The crucial Brin-Thompson groups $sV$ for $s\in \N$, also known as higher dimensional Thompson's groups, were defined by Brin in \cite{Brin04} and proven to be of type $\F_\infty$ by Fluch-Marshler-Witzel-Zaremsky in \cite{Fluch13}. This infinite family of groups also has a braided variant family, known as the braided Brin-Thompson groups which we denote by $s\Vbr$. This new infinite family of groups can be defined in a similar way as was done to get from $V$ to $\Vbr$ and will be the main family of groups we focus on in this paper. We will start by constructing the groups $s\Vbr$ and proceed to prove that they are of type $\F_\infty$. 

\begin{maintheorem}\label{thrm:maintheorem}
The braided Thompson's groups $s\Vbr$ is of type $\F_\infty$.
\end{maintheorem}

Therefore the $s\Vbr$ provide an infinite family of groups generalizing $\Vbr$, in particular this is an infinite family of ``braided'' Thompson-like groups that are all of type $\F_\infty$. Our approach follows the proof in \cite{Bux16} that $\Vbr$ is of type $\F_\infty$, though a number of new ideas are required. In particular, the descending links arising in the study of the local structure are modeled by \emph{multicolored matching complexes on a surface}, which may be of independent interest. These complexes are given by a graph together with a surface containing the vertices of the graph, and consist of multicolored arc systems that yield a matching of the graph. We show these complexes to be highly connected in certain cases.

\begin{theorem*}[Theorem \ref{thrm:multicolored_surface_matching_conn}]
The multicolored matching complex on a surface for the complete graph on $n$ vertices is $(\lfloor \frac{n-2}{3} \rfloor-1)$-connected. 
\end{theorem*}

The Main Theorem can be viewed as part of a general attempt to
understand how the finiteness properties of a group change when it is braided. Another instance of this question concerns the \emph{braided Houghton groups} $BH_n$. In \cite{genevois_20} Genevois-Lonjou-Urech prove that, for any $n$, $BH_n$ is of type $F_{n-1}$ but not of type $F_n$. Another example involves braiding the Higman-Thompson groups $V_{n,r}$ and certain Nekrahevych groups of the form $V_{n,r}(H)$, and inspecting questions of finite generation, details of which can be found in \cite{aroca_20}.

In Section \ref{sec:braided_sV} we introduce the definition of $s\Vbr$, using the language similar to how $\Vbr$ was discussed in \cite{Bux16}. We also introduce ``multicolored spraiges'', or ``multicolored split-braid-merge diagrams.'' The Stein complex $sX_{br}$ is constructed in Section \ref{sec:def_stein_space} along with an invariant, cocompact filtration $(sX_{br}^{\le n})_{n\in\N}$. In Section \ref{sec:surfaces} multicolored matching complexes on surfaces are introduced and shown to be highly connected. These connectivity results are then used in Section \ref{sec:desc_link_conn} to show that the filtration~$(sX_{br}^{\le n})_{n\in\N}$ is highly connected. Finally we prove the Main Theorem in Section \ref{sec:proof_main_theorem}.


\subsection*{Acknowledgments}

The author is grateful to his Ph.D. advisor Matt Zaremsky for guiding him through the strategy to handle the complexes $s\matcharc(K_n)$ in Section \ref{sec:surfaces} and the filtration in Section \ref{sec:desc_link_conn}, especially the assistance with proofs of Lemmas \ref{lem:desc_morse_link_contra}, \ref{lem:bottleneck_merge}, \ref{lem:remove_simplex_connectivity}, and \ref{lem:boundary_inclusion_nullhomotopic}, as well as countless other circumstances during the writing of this paper. We also thank Matt Brin for providing the idea of braiding the previously researched Brin-Thompson groups, also known as Higher Dimensional Thompson groups $sV$ for $s \in \N$, and discussing the main theorem with the author.


\section{The braided Brin-Thompson groups}
\label{sec:braided_sV}

Thompson's groups and their variants have been studied at length, and possess many unusual and interesting properties. An introduction to the classical Thompson's groups can be found in \cite{cannon96}, in particular Thompson's group $V$ is an infinite simple group of type $\F_\infty$. A good introduction to the braided Thompson group $\Vbr$ is \cite{brady08} which is a group of type $\F_\infty$ as proven in \cite{Bux16}. Another example of a variant is the Brin-Thompson groups $sV$ introduced by Brin in \cite{Brin04} which are a family of infinite simple groups which are proven to be of type $\F_\infty$ in \cite{Fluch13}. The family of groups we will define in this section will generalize $V$ and $sV$, namely by adding colors and braids to the trees. We will picture elements of this new family of groups in the language of \emph{strand diagrams}, as in \cite{belk07}.

\subsection{The group}\label{sec:braided_sV_setup}

We first recall and introduce some important terminology. By a \emph{rooted binary tree} we mean a finite tree such that every vertex has degree $3$, except the \emph{leaves}, which have degree $1$, and the \emph{root}, which has degree $2$ (or degree $1$ if the root is also a leaf).  Usually we draw trees with the root at the top and the nodes descending from it, down to the leaves.  A non-leaf node together with the two nodes directly below it is called a \emph{caret}. By a \emph{rooted binary multicolored tree} we mean a rooted binary tree where the carets have a designated color. If the leaves of a caret in $T$ are leaves of $T$, we will call the caret \emph{elementary}. Note that a rooted binary multicolored tree always consists of $(n-1)$ carets and has $n$ leaves, for some $n \in \N$. These trees are important in defining the elements of our group, however, first we need to define some types of diagrams. At times it may be necessary to emphasize how many colors a diagram has. In such a case we will replace ``multicolored'' by ``$s$-colored'' where $s \geq 1$ is the number of colors needed. 

\begin{definition}
By a \emph{paired multicolored tree diagram} we mean a triple $(T_-,\rho,T_+)$ consisting of two rooted binary multicolored trees $T_-$ and $T_+$ with the same number of leaves $n$, and a permutation $\rho\in S_n$.  The leaves of $T_-$ are labeled $1,\dots,n$ from left to right, and for each $i$, the $\rho(i)^{\text{th}}$ leaf of $T_+$ is labeled $i$.  
\end{definition}

We can draw a paired multicolored tree diagram $(T_-, \rho, T_+)$ in two ways as follows. The first way would be as an ordered pair of the two trees labeled appropriately using the permutation. The second way would be to draw $T_-$ as described, $T_+$ upside down and below $T_-$, and arrows representing the permutation $\rho$. This way of drawing paired multicolored tree diagrams is closely related to the \emph{strand diagrams} model as done in \cite{belk07}. See Figure \ref{fig:reduction_sV} for an example of a drawing the first way and Figure \ref{fig:element_of_2V} for an example of drawing the second way. 

\begin{figure}[t]
	\centering
	\begin{tikzpicture}[line width=1pt, scale=0.7]
		\draw[blue]
		(-1.5,-1.5) -- (0,0) -- (1.5,-1.5)
		(1,-1) -- (0,-2);
		\draw[red]
		(-2,-2) -- (-1.5,-1.5)
		(1.5,-1.5) -- (2,-2)
		(-1.5,-1.5) -- (-1,-2)
		(1.5,-1.5) -- (1,-2);
		\filldraw[red]
		(1.5,-1.5) circle (1.5pt)
		(-2,-2) circle (1.5pt)
		(-1,-2) circle (1.5pt)
		(-1.5,-1.5) circle (1.5pt)
		(2,-2) circle (1.5pt)
		(1,-2) circle (1.5pt);
		\filldraw[blue]
		(0,0) circle (1.5pt)
		(0,-2) circle (1.5pt)
		(1,-1) circle (1.5pt);
		\node at (-2,-2.5) {$1$};
		\node at (-1,-2.5) {$2$};
		\node at (0,-2.5) {$3$};
		\node at (1,-2.5) {$4$};
		\node at (2,-2.5) {$5$};
		
		\begin{scope}[xshift=5.5cm, yscale=-1]
			\draw[red]
			(-1,1) -- (0,0) -- (2,2)
			(-.5,.5) -- (1,2)
			(-1,2) -- (-.5, 1.5)
			(-.5, 1.5) -- (0, 2);
			\draw[blue]
			(-2,2) -- (-1,1) -- (-.5,1.5);
			\filldraw[red]
			(0,0) circle (1.5pt)
			(0,2) circle (1.5pt)
			(2,2) circle (1.5pt)
			(-1,2) circle (1.5pt)
			(-.5, 1.5) circle (1.5pt)
			(-.5,.5) circle (1.5pt)
			(1,2) circle (1.5pt);
			\filldraw[blue]
			(-1,1) circle (1.5pt)
			(-2,2) circle (1.5pt);
			\node at (-2,2.5) {$3$};
			\node at (-1,2.5) {$1$};
			\node at (0,2.5) {$2$}; 
			\node at (1,2.5) {$5$};
			\node at (2,2.5) {$4$};
		\end{scope}
		
		\begin{scope}[xshift=-1cm, yshift=-4cm]
			\draw[blue]
			(0,-1.5) -- (1.5,0) -- (2.5,-1)
			(2,-.5)  -- (1,-1.5); 
			\draw[red]
			(2.5,-1) -- (3,-1.5)
			(2,-1.5) -- (2.5,-1);
			\filldraw[red]
			(2.5,-1) circle (1.5pt)
			(3,-1.5) circle (1.5pt)
			(2,-1.5) circle (1.5pt)
			(2.5,-1) circle (1.5pt)
			(2.5,-1) circle (1.5pt);
			\filldraw[blue]
			(1.5,0) circle (1.5pt)
			(0,-1.5) circle (1.5pt)
			(1,-1.5) circle (1.5pt)
			(2,-.5) circle (1.5pt);
			\node at (0,-2) {$1$};
			\node at (1,-2) {$2$};
			\node at (2,-2) {$3$};
			\node at (3,-2) {$4$};
		\end{scope}
		
		\begin{scope}[xshift=3.5cm, yshift=-5.5cm, yscale=-1]
			\draw[red]
			(.5,-.5) -- (1.5,-1.5) -- (3,0)
			(1,-1) -- (2,0);
			\draw[blue]
			(0,0) -- (.5,-.5)
			(.5,-.5) -- (1,0);
			\filldraw[red]
			(1.5,-1.5) circle (1.5pt)
			(3,0) circle (1.5pt)
			(1,-1) circle (1.5pt)
			(3,0) circle (1.5pt)
			(2,0) circle (1.5pt);
			\filldraw[blue]
			(0,0) circle (1.5pt)
			(.5,-.5) circle (1.5pt)
			(1,0) circle (1.5pt);
			\node at (0,0.5) {$2$}; 
			\node at (1,0.5) {$1$};
			\node at (2,0.5) {$4$}; 
			\node at (3,0.5) {$3$};
		\end{scope}
	\end{tikzpicture}
	\caption{Reduction of the top paired multicolored tree diagram to the bottom one.}
	\label{fig:reduction_sV}
\end{figure}
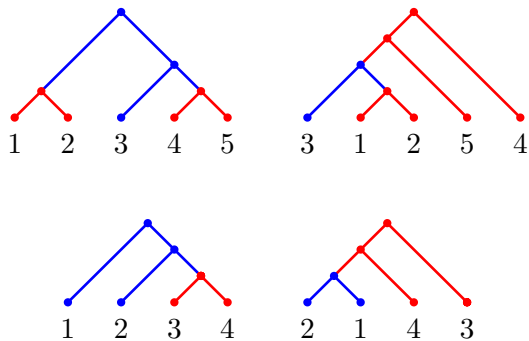

\begin{figure}[t]
	\centering
	\begin{tikzpicture}[line width=1pt,scale=0.9]
		\draw[blue]
		(-1.5,-1.5) -- (0,0) -- (1.5,-1.5)
		(1,-1) -- (0,-2);
		\draw[red]
		(-2,-2) -- (-1.5,-1.5)
		(1.5,-1.5) -- (2,-2)
		(-1.5,-1.5) -- (-1,-2)
		(1.5,-1.5) -- (1,-2);
		\filldraw[red]
		(1.5,-1.5) circle (1.5pt)
		(-2,-2) circle (1.5pt)
		(-1,-2) circle (1.5pt)
		(-1.5,-1.5) circle (1.5pt)
		(2,-2) circle (1.5pt)
		(1,-2) circle (1.5pt);
		\filldraw[blue]
		(0,0) circle (1.5pt)
		(0,-2) circle (1.5pt)
		(1,-1) circle (1.5pt);
		\draw[->](-1.9,-2.1) -> (-1.1,-2.9); \draw[->](-0.9,-2.1) ->
		(-0.1,-2.9); \draw[->](-0.1,-2.1) -> (-1.9,-2.9); \draw[->](1.1,-2.1) ->
		(1.9,-2.9); \draw[->](1.9,-2.1) -> (1.1,-2.9);
		
		\begin{scope}[yshift=-5cm]
			\draw[red]
			(-1,1) -- (0,0) -- (2,2)
			(-.5,.5) -- (1,2)
			(-1,2) -- (-.5, 1.5)
			(-.5, 1.5) -- (0, 2);
			\draw[blue]
			(-2,2) -- (-1,1) -- (-.5,1.5);
			\filldraw[red]
			(0,0) circle (1.5pt)
			(0,2) circle (1.5pt)
			(2,2) circle (1.5pt)
			(-1,2) circle (1.5pt)
			(-.5, 1.5) circle (1.5pt)
			(-.5,.5) circle (1.5pt)
			(1,2) circle (1.5pt);
			\filldraw[blue]
			(-1,1) circle (1.5pt)
			(-2,2) circle (1.5pt);
		\end{scope}
	\end{tikzpicture}
	\caption{An element of $2V$.}
	\label{fig:element_of_2V}
\end{figure}
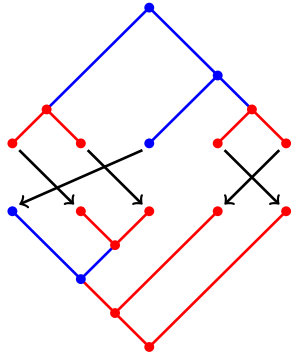

There exists an equivalence relation on paired multicolored tree diagrams given by reductions, expansions, and application of the cross relation. By a \emph{reduction} we mean the following: Suppose there is an elementary caret in $T_-$ with left leaf labeled $i$ and right leaf labeled $i+1$, and an elementary caret in $T_+$ of the same color with left leaf labeled $i$ and right leaf labeled $i+1$.  Then we can ``reduce'' the diagram by removing those carets, renumbering the leaves and replacing $\rho$ with the permutation $\rho'\in S_{n-1}$ that sends the new leaf of $T_-$ to the new leaf of $T_+$, and otherwise behaves like $\rho$. The resulting paired multicolored tree diagram $(T'_-,\rho',T'_+)$ is then said to be obtained by \emph{reducing} $(T_-,\rho, T_+)$. The reverse operation to reduction is called \emph{expansion}, so $(T_-,\rho,T_+)$ is an expansion of $(T'_-,\rho',T'_+)$. See Figure \ref{fig:reduction_sV} for an idea of reduction of paired multicolored tree diagrams. The additional relation that can be applied is called the \textit{cross relation} and is defined below. Due to this additional relation, ``reduced'' representatives of paired multicolored tree diagrams are not unique.

\begin{definition}[Cross Relation]
	Consider the paired $2$-colored tree diagram where the first tree $T_-$ consists of a red caret with a blue caret on each leaf, the second tree consists of a blue caret with a red caret on each leaf, and the permutation $\rho = (23)$. This paired rooted multicolored tree diagram is equivalent to the trivial one. This is defined to be the \textit{cross relation} and can be generalized to any number of dimensions. See Figure \ref{fig:2Vcrossrelation} for a visualization of this relation.
\end{definition}

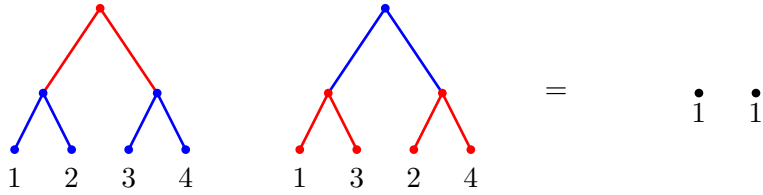
\begin{figure}[t]
	\centering
	\begin{tikzpicture}[line width=1pt, scale=.75]
		
		\begin{scope}[xshift=-3cm,yshift=-1cm]
			\coordinate (a) at (0,0);
			\coordinate (b) at ($(a) + (1,0)$);
			\coordinate (c) at ($(a) + (2,0)$);
			\coordinate (d) at ($(a) + (3,0)$);
			\coordinate (e) at ($(a) + (0.5,1)$);
			\coordinate (f) at ($(a) + (2.5,1)$);
			\coordinate (g) at ($(a) + (1.5,2.5)$);
			\coordinate (h) at ($(a) + (5,0)$);
			\coordinate (i) at ($(a) + (6,0)$);
			\coordinate (j) at ($(a) + (7,0)$);
			\coordinate (k) at ($(a) + (8,0)$);
			\coordinate (l) at ($(a) + (5.5,1)$);
			\coordinate (m) at ($(a) + (7.5,1)$);
			\coordinate (n) at ($(a) + (6.5,2.5)$);
			\coordinate (a') at ($(a) + (0,-1)$);
			\coordinate (b') at ($(b) + (0,-1)$);
			\coordinate (c') at ($(c) + (0,-1)$);
			\coordinate (d') at ($(d) + (0,-1)$);
			
			\draw[line width=1pt, blue]
			(a) -- (e) -- (b)
			(c) -- (f) -- (d)
			(l) -- (n) -- (m);
			\draw[line width=1pt, red]
			(e) -- (g) -- (f)
			(h) -- (l) -- (i)
			(j) -- (m) -- (k);
			
			\filldraw[blue]
			(a) circle (1.5pt)
			(b) circle (1.5pt)
			(c) circle (1.5pt)
			(d) circle (1.5pt)
			(e) circle (1.5pt)
			(f) circle (1.5pt)
			(n) circle (1.5pt); 
			\filldraw[red] 
			(g) circle (1.5pt)
			(h) circle (1.5pt)
			(i) circle (1.5pt)
			(j) circle (1.5pt)
			(k) circle (1.5pt)
			(l) circle (1.5pt)
			(m) circle (1.5pt);  
			 
			\node at (0, -0.5) {1};
			\node at (1, -0.5) {2};
			\node at (2, -0.5) {3};
			\node at (3, -0.5) {4};
			\node at (5, -0.5) {1};
			\node at (6, -0.5) {3};
			\node at (7, -0.5) {2};
			\node at (8, -0.5) {4};
		\end{scope}
		\node at (6.5,0) {$=$};
		
		\begin{scope}[xshift=9cm, yshift=-1cm]
			\coordinate (a) at (0,1);
			\coordinate (b) at (1,1);
			
			
			\filldraw
			(a) circle (1.5pt)
			(b) circle (1.5pt);
			
			\node at (0, 0.65) {1};
			\node at (1, 0.65) {1};
		\end{scope}
	\end{tikzpicture}
	\caption{The cross relation in $2V$}
	\label{fig:2Vcrossrelation}
\end{figure}

There is a binary operation $\ast$ on the set of equivalence classes
of paired multicolored tree diagrams. Let $T=(T_-,\rho,T_+)$ and $S=(S_-,\xi,S_+)$ be paired multicolored tree diagrams.  By applying repeated expansions and possibly using the cross relation we can find representatives $(T'_-,\rho',T'_+)$ and $(S'_-,\xi',S'_+)$ of the equivalence classes of $T$ and $S$, respectively, such that $T'_+ = S'_-$.  Then we declare $T\ast S$ to be $(T'_-,\rho'\xi',S'_+)$. This operation is well defined on the equivalence classes, and is a group operation. 

\begin{definition}[Brin-Thompson group $sV$]
	\label{def:sV}
	The Brin-Thompson group $sV$ ($s \in \N$) is the group of equivalence classes of paired $s$-colored tree diagrams with the multiplication $\ast$.
\end{definition}

Observe that while the above definition is in terms of $s$-colored tree diagrams with the group operation $\ast$, this is different than how the group was first defined in \cite{Brin04}. The $s$-colored trees represent a subdivision of a $s$-cube, and so the $s$-colored tree diagrams represent a map from one subdivided $s$-cube to another. This way of viewing $sV$ has a group operation of composition as its elements are maps. More details on this model for $sV$ can be found in \cite{Brin04} and \cite{Fluch13}. Moreover, more details on the model for which we need in this paper can be found in \cite{Burillo08}.

Now we construct the braided version $s\Vbr$ of $sV$. The idea is to replace the permutations of leaves by braids between the leaves. We will first introduce \emph{braided paired multicolored tree diagrams} and then copy the construction of $sV$ given above to define $s\Vbr$.  

\begin{definition}
	A \emph{braided paired multicolored tree diagram} is a triple $(T_-,b,T_+)$ consisting of two rooted binary multicolored trees $T_-$ and $T_+$ with the same number of leaves $n$ and a braid $b \in B_n$. 
\end{definition}

As with $sV$, we can define an equivalence relation on the set of
braided paired multicolored tree diagrams using the notions of reduction and expansion plus a braided cross relation. We will first define expansion and then take reduction as the reverse of expansion. The notions of expansions and reductions are all taken from the $s=1$ case in \cite{Bux16} and we add the braided cross relation which is new to this group. Let $\rho_b\in S_n$ denote the permutation corresponding to the braid $b\in B_n$. Let $(T_-,b,T_+)$ be a braided paired multicolored tree diagram.  Label the leaves of $T_-$ from $1$ to $n$, left to right, and for each $i$ label the $\rho_b(i)^{\text{th}}$ leaf of $T_+$ by $i$.  By the $i^{\text{th}}$ strand of the braid we will always mean the strand that begins at the $i^{\text{th}}$ leaf of $T_-$, i.e., we count the strands from the top.  An \emph{expansion} of $(T_-,b,T_+)$ amounts to the following. For some $1\le i\le n$, replace~$T_\pm$ with trees $T_\pm'$ obtained from $T_\pm$ by adding a caret to the leaf labeled~$i$.  Then replace $b$ with a braid $b' \in B_{n+1}$, obtained by ``doubling'' the $i^{\text{th}}$ strand of $b$.  The triple $(T_-',b',T_+')$ is an \emph{expansion} of $(T_-,b,T_+)$. As with paired multicolored tree diagrams, \emph{reduction} is the reverse of expansion, so $(T_-,b,T_+)$ is a reduction of $(T_-',b',T_+')$.  See Figure \ref{fig:reduction_sVbr} for an idea of reduction of braided paired multicolored tree diagrams. The additional relation that can be applied is called the \textit{braided cross relation} and is defined below. Due to this additional relation, ``reduced'' representatives of paired multicolored tree diagrams are not unique. The equivalence class of $(T_-,b,T_+)$ will be denoted by $[(T_-,b,T_+)]$, or for simplicity sometimes by $[T_-,b,T_+]$.

\begin{figure}
	\centering
	\begin{tikzpicture}[line width=1pt,scale=0.75]
		\draw[blue]
		(.5,-1.5) -- (2,0) -- (3.5,-1.5)
		(3,-1) -- (2,-2);
		\draw[red]
		(0,-2) -- (.5,-1.5)
		(3.5,-1.5) -- (4,-2)
		(0.5,-1.5) -- (1,-2)
		(3.5,-1.5) -- (3,-2);
		\draw
		(2,-2) to [out=-90, in=90] (0,-4)   
		(3,-2) to [out=-90, in=90] (4,-4);
		\draw[white, line width=4pt]
		(0,-2) to [out=-90, in=90] (1,-4)   
		(1,-2) to [out=-90, in=90] (2,-4)  
		(4,-2) to [out=-90, in=90] (3,-4);
		\draw
		(0,-2) to [out=-90, in=90] (1,-4)   
		(1,-2) to [out=-90, in=90] (2,-4)
		(4,-2) to [out=-90, in=90] (3,-4);
		\draw[red]
		(1,-5) -- (2,-6) -- (4,-4)
		(1.5,-5.5) -- (3,-4)
		(1,-4) -- (1.5,-4.5)
		(1.5, -4.5) -- (2, -4);
		\draw[blue]
		(0,-4) -- (1,-5) -- (1.5,-4.5);
		
		\filldraw[red]
		(0,-2) circle (1.5pt)   
		(4,-2) circle (1.5pt)
		(0.5,-1.5) circle (1.5pt)
		(1,-2) circle (1.5pt)
		(3.5,-1.5) circle (1.5pt) 
		(3,-2) circle (1.5pt);
		\filldraw[blue]
		(3,-1) circle (1.5pt)  
		(2,-2) circle (1.5pt)   
		(2,0) circle (1.5pt);
		
		\filldraw[blue]
		(1,-5) circle (1.5pt)   
		(0,-4) circle (1.5pt);
		\filldraw[red]
		(1.5,-5.5) circle (1.5pt)
		(4,-4) circle (1.5pt)  
		(2,-6) circle (1.5pt)   		
		(3,-4) circle (1.5pt)   
		(1,-4) circle (1.5pt)  
		(2,-4) circle (1.5pt)
		(1.5,-4.5) circle (1.5pt);
		
		\draw[->]
		(5,-3) -> (6,-3);
		
		\begin{scope}[xshift=7cm,yshift=-1cm]
			\draw[blue]
			(0,-1.5) -- (1.5,0) -- (2.5,-1)
			(2,-.5)  -- (1,-1.5); 
			\draw[red]
			(2.5,-1) -- (3,-1.5)
			(2,-1.5) -- (2.5,-1);
			\draw
			(1,-1.5) to [out=-90, in=90] (0,-3)   
			(2,-1.5) to [out=-90, in=90] (3,-3);
			\draw[white, line width=4pt]
			(0,-1.5) to [out=-90, in=90] (1,-3)   
			(3,-1.5) to [out=-90, in=90] (2,-3);
			\draw
			(0,-1.5) to [out=-90, in=90] (1,-3)   
			(3,-1.5) to [out=-90, in=90] (2,-3);
			
			\draw[red]
			(.5,-3.5) -- (1.5,-4.5) -- (3,-3)
			(1,-4) -- (2,-3);
			\draw[blue]
			(0,-3) -- (.5,-3.5)
			(.5,-3.5) -- (1,-3);
			
			\filldraw[red]
			(2,-1.5) circle (1.5pt) 
			(3,-1.5) circle (1.5pt)  
			(2.5,-1) circle (1.5pt);
			\filldraw[blue]
			(0,-1.5) circle (1.5pt)   
			(1.5,0) circle (1.5pt)  
			(1,-1.5) circle (1.5pt)     
			(2,-.5) circle (1.5pt);
			
			\filldraw[red]
			(1.5,-4.5) circle (1.5pt)   
			(3,-3) circle (1.5pt)
			(1,-4) circle (1.5pt)   
			(2,-3) circle (1.5pt);
			\filldraw[blue]
			(.5, -3.5) circle (1.5pt) 
			(0,-3) circle (1.5pt)   
			(1,-3) circle (1.5pt);
		\end{scope}
	\end{tikzpicture}
	\caption{Reduction of braided paired multicolored tree diagrams.}
	\label{fig:reduction_sVbr}
\end{figure}
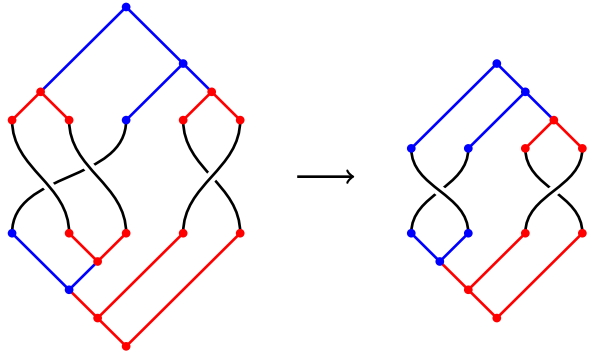

\begin{definition}[Braided cross relation]
	Consider the cross relation in $2V$. Replace the permutation with a braid as done in Figure \ref{fig:2Vbrcrossrelation} . This is known as the \textit{braided cross relation} and can be generalized to $s$ dimensions. Fix an arbitrary total ordering of the colors. The convention in the braid of left-over-right is for when red is less than blue. If red is greater than blue, then we braid right-over-left.
\end{definition}

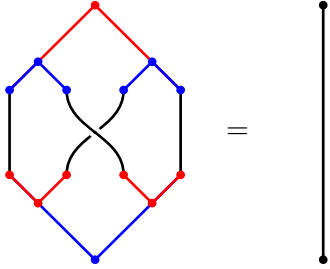
\begin{figure}[t]
	\centering
	\begin{tikzpicture}[line width=1pt, scale=.75]
		
		\begin{scope}[xshift=7cm,yshift=-1cm]
			\draw[line width=1pt, red]
			(0,-1.5) -- (1.5,0);
			\draw[line width=1pt, red]
			(1.5,0) -- (3,-1.5);
			\draw[line width=1pt, blue]
			(2.5,-1) -- (3,-1.5);
			\draw[line width=1pt, blue]
			(0.5,-1) -- (1,-1.5); 
			\draw[line width=1pt, blue]
			(0.5,-1) -- (0,-1.5); 
			\draw[line width=1pt, blue]
			(2,-1.5) -- (2.5,-1);
			\draw
			(1,-1.5) to [out=-90, in=90] (2,-3);
			\draw 
			(2,-1.5) to [out=-90, in=90] (1,-3);
			\draw[white, line width=1.5pt]
			(1.1,-1.5) to [out=-90, in=90] (2.1,-3);
			\draw[white, line width=1.5pt]
			(.9,-1.5) to [out=-90, in=90] (1.9,-3);
			\draw
			(0,-1.5) to [out=-90, in=90] (0,-3);
			\draw
			(3,-1.5) to [out=-90, in=90] (3,-3);
			\draw[blue]
			(0,-3) -- (1.5,-4.5);  
			\draw[blue]
			(1.5,-4.5) -- (3,-3);  
			\draw[red]
			(2.5,-3.5) -- (2,-3);
			\draw[red]
			(2.5,-3.5) -- (3,-3);
			\draw[red]
			(.5,-3.5) -- (0,-3);
			\draw[red]
			(.5,-3.5) -- (1,-3);
			
			\filldraw[blue]
			(0,-1.5) circle (1.5pt); 
			\filldraw[red] 
			(1.5,0) circle (1.5pt);  
			\filldraw[blue]
			(3,-1.5) circle (1.5pt) ;
			\filldraw[blue]
			(0.5,-1) circle (1.5pt) (1,-1.5) circle (1.5pt);
			\filldraw[blue]
			(2,-1.5) circle (1.5pt)  
			(2.5,-1) circle (1.5pt);
			\filldraw[red]
			(0,-3) circle (1.5pt);
			\filldraw[blue]
			(1.5,-4.5) circle (1.5pt);
			\filldraw[red]
			(0.5,-3.5) circle (1.5pt);
			\filldraw[red]
			(2.5,-3.5) circle (1.5pt);
			\filldraw[red]
			(3,-3) circle (1.5pt);
			\filldraw[red]
			(2,-3) circle (1.5pt);
			\filldraw[red]
			(1,-3) circle (1.5pt);
		\end{scope}
		\node at (11,-3.25) {$=$};
		
		\begin{scope}
			\draw
			(12.5,-1) -- (12.5, -5.5);
			
			\filldraw
			(12.5,-1) circle (1.5pt);
			\filldraw
			(12.5,-5.5) circle (1.5pt);
		\end{scope}
	\end{tikzpicture}

	\caption{Example of a braided equivalence relation in $2\Vbr$}
	\label{fig:2Vbrcrossrelation}
\end{figure}

Two braided paired multicolored tree diagrams are equivalent if one is obtained from the other by a finite sequence of reductions, expansions, or applications of braided cross relations. We will now discuss a binary operation $\ast$ on the set of equivalence classes of braided paired multicolored tree diagrams. Let $T = (T_-, b, T_+)$ and $S=(S_-, c, S_+)$ be reduced braided paired multicolored tree diagrams. We can find representatives $(T_{-}', b', T_{+}')$ and $(S_{-}', c', S_{+}')$ of the equivalence classes of $T$ and $S$, respectively, such that $T_{+}' = S_{-}'$. This may require the use of the braided cross relation defined above which was not the case for $\Vbr$. Then we declare $T \ast S$ to be $(T_{-}', b'c', S_{+}')$. We will proceed to show that this operation is well defined on the equivalence classes, and is a group operation. First we will show that it is well defined for reductions and expansions, following with it being well defined for braided cross relations.

Take $[T_-,b,T_+]=[T_-',b',T_+']$ and $[U_-,c,U_+]$ such that without loss of generality $T_+ = U_-$. Moreover, the first equality in the previous sentence holds for the entire equivalence class of $(T_-, b, T_+)$, so we may take $T_-'$ to be $T_-$ with one new caret added to the $k^{th}$ leaf, $b'$ is $b$ where we double the $k^{th}$ strand, and $T_+'$ has a new caret of the same caret on the $l^{th}$ leaf where the bottom of strand which was doubled in $b$ is the $l^{th}$ strand. Observe that $b'$ and $T_+'$ are determined by $T_-'$. We need to show that $[T_-,b, T_+][U_-,c,U_+]=[T_-',b', T_+'][U_-,c,U_+]$. By definition we know that $[T_-,b, T_+][U_-,c,U_+] = [T_-,bc, U_+]$ and we can find $[T_-',b', T_+'][U_-,c, U_+]$ by expanding $U_-$ to $U_-'$ by adding a single caret of the same color as was added to $T_-'$ to the $l^{th}$ leaf to make $T_+' = U_-'$ such that $c$ and $U_+$ turn into $c'$ and $U_+'$ appropriately. That is, we double the $l^{th}$ strand of $c$ to get $c'$. This will result in $[T_-',b',T_+'][U_-,c, U_+]= [T_-',b',T_+'][U_-',c', U_+']=[T_-',b'c', U_+']$. When we expand $(T_-,bc,U_+)$ in the same way as described above, we get $(T_-',(bc)',U_+')$ where $(bc)'$ is the result of doubling the $k^{th}$ strand of $bc$. Moreover, the bottom of the $k^{th}$ strand of $b$ is in the $l^{th}$ position of $c$, so when we double the $k^{th}$ strand of $bc$ we equivalently doubled the $k^{th}$ strand of $b$ and the $l^{th}$ strand of $c$. This is exactly $b'c'$. Therefore, $(bc)'$ is equal to $b'c'$. Hence we can expand $(T_-,bc, U_+)$ to get $(T_-', b'c', U_+')$, that is, $[T_-',b'c', U_+']=[T_-,bc, U_+]$. An analogous argument handles changing the representative of the right factor. 

Similarly, again take $[T_-,b,T_+]=[T_-',b',T_+']$ where $(T_-', b', T_+')$ is $(T_-, b, T_+)$ after one application of the expansion of a $k^{th}$ leaf of $T_-$, $k^{th}$ strand of $b$, and the $l^{th}$ leaf of $T_+$ by the braided cross relation as indicated by Figure \ref{fig:2Vbrcrossrelation}. As with the previous paragraph, also take $[U_-,c,U_+]$ where without loss of generality $T_+ = U_-$. Note the bottom of the $k^{th}$ strand of $b$ is the $l^{th}$ leaf of $T_+$. As before, $T_+ = U_-$ implies $[T_-,b,T_+][U_-,c,U_+] = [T_-,bc,U_+]$ and we can find $[T_-',b',T_+'][U_-,c,U_+]$ using the following argument. As $T_+'$ was obtained by adding the appropriate combination of carets to the $l^{th}$ leaf of $T_+$, namely a blue caret with a red caret on each foot, we  apply the braided cross relation upside down to the $l^{th}$ leaf of $U_-$ of $(U_-,c,U_+)$ to get $(U_-',c',U_+')$. It follows that $T_+' = U_-'$, hence $[T_-',b',T_+'][U_-',c',U_+'] = [T_-',b'c',U_+']$. Observe that the additional braid in $c'$ is the inverse of the additional braid in $b'$, so these recently added braids cancel each other out and can be replaced with the trivial braid on four strands. From there we can reduce $T_-'$ and $U_+'$ by removing the recently added carets. In other words $(T_-',b'c',U_+')$ reduces to $(T_-,bc,U_+)$. Therefore, $[T_-',b'c',U_+'] = [T_-,bc,U_+]$ as desired.
	
The operation $\ast$ on the set of equivalence classes of braided paired multicolored tree diagrams is therefore a well defined group operation.

\begin{definition}[Braided Brin-Thompson group $s\Vbr$]
	The braided Brin-Thompson group $s\Vbr$ ($s \in \N$) is the group of equivalence classes of braided paired $s$-colored tree diagrams with multiplication $\ast$.
\end{definition}

From now on we will just refer to the braided paired multicolored tree diagrams as being the elements of $s\Vbr$, though one should keep in mind that the elements are actually equivalence classes under the reduction, expansion, and braided cross relation operations.

It is convenient to have the following picture in mind for elements of $s\Vbr$. We can think of $T_-$ as described and $T_{+}$  is drawn beneath it upside down so that the root is at the bottom and the leaves are at the top. We then connect the leaves of $T_-$ to the leaves of $T_+$ as indicated by the braid $b$. This turns the picture into a kind of strand diagram as in \cite{belk07}. Notice the similarity of drawing elements of $s\Vbr$ to drawing elements of $sV$ as demonstrated in Figure \ref{fig:element_of_2V}. In other words we are taking an element of $sV$ and replacing the permutation with a braid. See the previously mentioned Figure \ref{fig:reduction_sVbr} for an example of an element in $2\Vbr$.

With this way of picturing elements of $s\Vbr$, we have a convenient way to visualize multiplication in $s\Vbr$ via ``stacking'' braided paired multicolored tree diagrams. For $g,h\in s\Vbr$, each pictured as described before, $g\ast h$ is obtained by attaching the top of $h$ to the bottom of $g$ and then reducing the picture via certain moves. We indicate four of these moves in Figure \ref{fig:reduction_moves}. A ``merge'' which is a caret opening upwards followed immediately by a ``split'' which is a caret opening downwards of the same color, or a split followed immediately by a merge of the same color, is equivalent to doing nothing, as seen in the top two pictures. Also, splits and merges interact with braids in ways indicated by the bottom two pictures. Additionally, we may need to apply the braided cross relation in order to obtain $g \ast h$. We leave it to the reader to further inspect the details of this visualization of multiplication in $s\Vbr$.

\begin{figure}[t]
	\centering
	\begin{tikzpicture}[line width=1pt]
		
		\begin{scope}[xshift=-2.75cm, yshift=-0.25cm]	
			\coordinate (a) at (0,0);
			\coordinate (b) at ($(a)+(0,0.75)$);
			\coordinate (c) at ($(a)+(1.5,0.75)$);
			\coordinate (d) at ($(a)+(0,-0.75)$);
			\coordinate (e) at ($(a)+(1.5,-0.75)$);
			\coordinate (f) at ($(a)+(0.75,0)$);
			
			\draw[red]
			(b) -- (e)
			(c) -- (d);
			
			\filldraw[red]
			(b) circle (1pt)
			(c) circle (1pt)
			(d) circle (1pt)
			(e) circle (1pt)
			(f) circle (1pt);
			
			\node at (2,0) {$=$};
		\end{scope}
		\begin{scope}[xshift=0.25cm, yshift=-0.25cm]
			\coordinate (a) at (0,0);
			\coordinate (b) at ($(a)+(0,0.75)$);
			\coordinate (c) at ($(a)+(1.5,0.75)$);
			\coordinate (d) at ($(a)+(0,-0.75)$);
			\coordinate (e) at ($(a)+(1.5,-0.75)$);
			\coordinate (f) at ($(a)+(0.75,0)$);
			
			\draw
			(b) -- (d)
			(c) -- (e);
			
			\filldraw
			(b) circle (1pt)
			(c) circle (1pt)
			(d) circle (1pt)
			(e) circle (1pt);
		\end{scope}
		\begin{scope}[xshift=4.75cm]
			\coordinate (a) at (0,0);
			\coordinate (b) at ($(a)+(0.75,1.5)$);
			\coordinate (c) at ($(a)+(0.75,0.75)$);
			\coordinate (d) at ($(a)+(0.75,-0.75)$);
			\coordinate (e) at ($(a)+(1.5,0)$);
			\coordinate (f) at ($(a)+(0.75,-1.5)$);
			\coordinate (g) at ($(a)+(2.75,-1.5)$);
			\coordinate (h) at ($(a)+(2.75,1.5)$);
			\draw
			(b) -- (c)
			(d) -- (f);
			\draw[red]
			(a) -- (c) -- (e) -- (d) -- (a);  
			\node at (2.1,0) {$=$};
			\draw
			(g) -- (h);
			
			\filldraw[red]
			(a) circle (1pt)
			(c) circle (1pt)
			(d) circle (1pt)
			(e) circle (1pt);
			\filldraw
			(b) circle (1pt)
			(f) circle (1pt)
			(g) circle (1pt)
			(h) circle (1pt);
		\end{scope}
		\begin{scope}[yshift=-4.5cm,xshift=-3.5cm, scale=0.75pt]
			\coordinate (a) at (0,0);
			\coordinate (b) at ($(a)+(0.5,0.5)$);
			\coordinate (c) at ($(a)+(1,0)$);
			\coordinate (d) at ($(a)+(0.5,3)$);
			\coordinate (e) at ($(a)+(3,3)$);
			\coordinate (f) at ($(a)+(3,0)$);
			\coordinate (g) at ($(a)+(0.25,0.25)$);
			
			\draw[red]
			(a) -- (b) -- (c);
			\draw
			(d) to [out=-90, in=90, looseness=1] (f);
			\draw[white, line width=4pt]
			(b) to [out=90, in=-90, looseness=1] (e);
			\draw
			(b) to [out=90, in=-90, looseness=1] (e);
			\filldraw[red]
			(a) circle (1pt)
			(b) circle (1pt)
			(c) circle (1pt);
			\filldraw
			(d) circle (1pt)
			(e) circle (1pt)
			(f) circle (1pt);
			\node at (3.75,1.5) {$=$};
		\end{scope}
		
		\begin{scope}[yshift=-4.5cm, xshift=-.05cm, scale=0.75pt]
			\coordinate (a) at (0,0);
			\coordinate (b) at ($(a)+(0.5,0.5)$);
			\coordinate (c) at ($(a)+(1,0)$);
			\coordinate (d) at ($(a)+(0.5,3)$);
			\coordinate (e) at ($(a)+(3,3)$);
			\coordinate (f) at ($(a)+(3,0)$);
			\coordinate (g) at ($(a)+(2.5,2.5)$);
			\coordinate (h) at ($(a)+(3.5,2.5)$);
			\draw[red]
			(g) -- (e) -- (h);
			\draw
			(d) to [out=-90, in=90, looseness=1] (f);
			\draw[white, line width=4pt]
			(g) to [out=-90, in=90, looseness=1] (a)
			(h) to [out=-90, in=90, looseness=1] (c);
			\draw
			(g) to [out=-90, in=90, looseness=1] (a)
			(h) to [out=-90, in=90, looseness=1] (c);
			\filldraw
			(a) circle (1pt)
			(c) circle (1pt)
			(d) circle (1pt)
			(f) circle (1pt);
			\filldraw[red]
			(e) circle (1pt)
			(g) circle (1pt)
			(h) circle (1pt);
		\end{scope}
		\begin{scope}[yshift=-2.25cm,xshift=4cm, yscale=-1, scale=0.75pt]
			\coordinate (a) at (0,0);
			\coordinate (b) at ($(a)+(0.5,0.5)$);
			\coordinate (c) at ($(a)+(1,0)$);
			\coordinate (d) at ($(a)+(0.5,3)$);
			\coordinate (e) at ($(a)+(3,3)$);
			\coordinate (f) at ($(a)+(3,0)$);
			\coordinate (g) at ($(a)+(0.25,0.25)$);
			
			\draw[red]
			(a) -- (b) -- (c);
			\draw
			(d) to [out=-90, in=90, looseness=1] (f);
			\draw[white, line width=4pt]
			(b) to [out=90, in=-90, looseness=1] (e);
			\draw
			(b) to [out=90, in=-90, looseness=1] (e);
			\filldraw[red]
			(a) circle (1pt)
			(b) circle (1pt)
			(c) circle (1pt);
			\filldraw
			(d) circle (1pt)
			(e) circle (1pt)
			(f) circle (1pt);
			\node at (3.75,1.5) {$=$};
		\end{scope}
		
		\begin{scope}[yshift=-2.25cm, xshift=7.5cm, yscale=-1, scale=0.75pt]
			\coordinate (a) at (0,0);
			\coordinate (b) at ($(a)+(0.5,0.5)$);
			\coordinate (c) at ($(a)+(1,0)$);
			\coordinate (d) at ($(a)+(0.5,3)$);
			\coordinate (e) at ($(a)+(3,3)$);
			\coordinate (f) at ($(a)+(3,0)$);
			\coordinate (g) at ($(a)+(2.5,2.5)$);
			\coordinate (h) at ($(a)+(3.5,2.5)$);
			\draw[red]
			(g) -- (e) -- (h);
			\draw
			(d) to [out=-90, in=90, looseness=1] (f);
			\draw[white, line width=4pt]
			(g) to [out=-90, in=90, looseness=1] (a)
			(h) to [out=-90, in=90, looseness=1] (c);
			\draw
			(g) to [out=-90, in=90, looseness=1] (a)
			(h) to [out=-90, in=90, looseness=1] (c);
			\filldraw
			(a) circle (1pt)
			(c) circle (1pt)
			(d) circle (1pt)
			(f) circle (1pt);
			\filldraw[red]
			(e) circle (1pt)
			(g) circle (1pt)
			(h) circle (1pt);
		\end{scope}
	\end{tikzpicture}
	\caption{Moves to reduce braided paired multicolored tree diagrams after stacking. Note only carets have colors, not braids, and the color red can be interchanged with other colors as long as it stays consistent for both sides of the equality.}
	\label{fig:reduction_moves}
\end{figure}
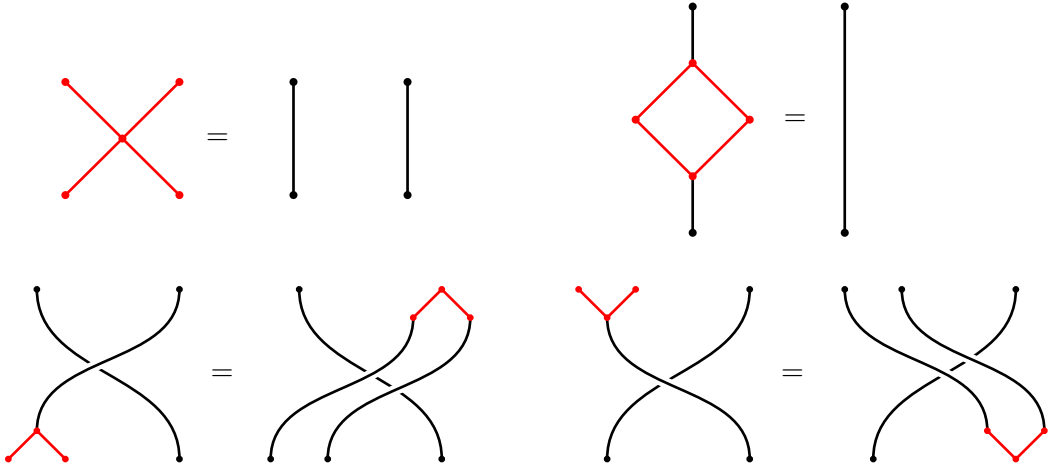

Our current goal is to inspect the higher finiteness properties of the $s\Vbr$, though first we will need some more language.  We now introduce a class of diagrams that will be used throughout the rest of this paper.

\subsection{A general class of diagrams}
\label{sec:general_diagrams}

As in \cite{Bux16}, to define the spaces we will use, we need a broader class of diagrams that generalizes braided paired multicolored tree diagrams, namely we will consider multicolored forests instead of multicolored trees. We will also continue to use the notion of strand diagrams. Here a multicolored forest will always mean a finite linearly ordered union of binary rooted multicolored trees. Given a braided paired multicolored tree diagram $(T_-,b,T_+)$ we call a caret in $T_-$ a \emph{split}. Similarly a \emph{merge} is a caret in $T_+$. With this terminology, we can call the picture representing the braided paired tree diagram a \emph{multicolored split-braid-merge diagram}, abbreviated \emph{multicolored spraige}. That is, we first picture one strand splitting up into $n$ strands in a certain way, representing $T_-$. Then the $n$ strands braid with each other, representing $b$. Finally according to $T_+$ we merge the strands back together. These special kinds of diagrams will also be called \emph{multicolored $(1,1)$-spraiges}.  More generally:

\begin{definition}[Multicolored Spraiges]
	\label{def:multi_spraiges}
	A \textit{multicolored} $(n,m)$-\textit{spraige} is a multicolored spraige that begins on $n$ strands, the \emph{(multicolored) heads}, and ends on $m$ strands, the \emph{(multicolored) feet}, and is braided in the middle. We can equivalently think of a \textit{multicolored} $(n,m)$-\textit{spraige} as a \textit{braided paired multicolored forest diagram} $(F_-, b, F_+)$, where $F_-$ has $n$ roots, $F_+$ has $m$ roots and both have the $k$ leaves with $b \in B_k$. By a \textit{multicolored} $n$-\textit{spraige} we mean a multicolored $(n,m)$-spraige for some $m$, and by a \textit{multicolored spraige} we mean a multicolored $(n,m)$-spraige for some $n$ and $m$. Let $s\spraige$ denote the set of all $s$-colored spraiges, $s\spraige_{n,m}$ denote the set of all $s$-colored $(n,m)$-spraiges, and $s\spraige_{n}$ denote the set of $s$-colored $n$-spraiges. 
\end{definition}

\begin{observation}[Left Cancellation]
	Note that if $\sigma\in\spraige_{n,m}$ and $\tau_1, \tau_2
	\in \spraige_{m,\ell}$ with $[\sigma \ast \tau_1] =
	[\sigma \ast \tau_2]$, then $[\tau_1]=[\tau_2]$. 
\end{observation}

Note that a multicolored $n$-spraige has $n$ heads, but can have any number of feet. Later we will need to use the ``number of feet'' function, which we define as $f\colon s\spraige\to\N$ given by $f(\sigma)=m$ if $\sigma\in s\spraige_{n,m}$ for some $n$.

The pictures in Figure \ref{fig:spraiges-multiplication} are examples of multicolored spraiges. As a remark, in the figures, a single-node tree will sometimes be elongated to an uncolored edge, for aesthetic reasons. One can generalize the notion of reduction, expansion, and the braided cross relation to arbitrary multicolored spraiges, and consider equivalence classes under these moves. We will just call an equivalence class of multicolored spraiges a multicolored spraige, so in particular we have the following observation: 

\begin{observation}
	$s\Vbr$ equals the set of multicolored $(1,1)$-spraiges $s\spraige_{1,1}$.
\end{observation}

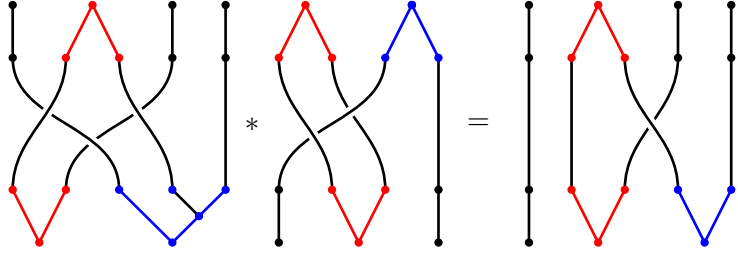
\begin{figure}[t]
\centering
\begin{tikzpicture}[line width=1pt,scale=0.7]
  
  \draw
   (0,0) -- (0,-1)   
   (3,0) -- (3,-1)
   (4,0) -- (4,-3.5);
   \draw[red]
   (1,-1) -- (1.5,0) -- (2,-1);
  \draw
   (3,-1) to [out=270, in=90] (1,-3.5);
  \draw[line width=4pt, white]
   (0,-1) to [out=270, in=90] (2,-3.5);
  \draw
   (0,-1) to [out=270, in=90] (2,-3.5);
  \draw[line width=4pt, white]
   (1,-1) to [out=270, in=90] (0,-3.5);
  \draw
   (1,-1) to [out=270, in=90] (0,-3.5);
  \draw[line width=4pt, white]
   (2,-1) to [out=270, in=90] (3,-3.5);
  \draw
   (2,-1) to [out=270, in=90] (3,-3.5) -- (3.5,-4);
  \draw[blue]
   (2,-3.5) -- (3,-4.5) -- (4,-3.5);
   \draw[red]
   (0,-3.5) -- (0.5,-4.5) -- (1,-3.5);
  \node at (4.5,-2.25) {$\ast$};
  \filldraw
  (0,0) circle (1.5pt)
  (3,0) circle (1.5pt)
  (4,0) circle (1.5pt)   
  (0,-1) circle (1.5pt)   
  (3,-1) circle (1.5pt)   
  (4,-1) circle (1.5pt);  
  \filldraw[red]
   (1,-1) circle (1.5pt)   
   (2,-1) circle (1.5pt)   
   (1.5,0) circle (1.5pt)      
   (0,-3.5) circle (1.5pt)   
   (1,-3.5) circle (1.5pt)   
   (0.5,-4.5) circle (1.5pt);
  \filldraw[blue]
   (2,-3.5) circle (1.5pt)  
   (3,-3.5) circle (1.5pt)   
   (4,-3.5) circle (1.5pt)
   (3.5,-4) circle (1.5pt)   
   (3,-4.5) circle (1.5pt);

 \begin{scope}[xshift=5cm]
  \draw[red]
   (0,-1) -- (0.5,0) -- (1,-1);
   \draw[blue]
   (2,-1) -- (2.5,0) -- (3,-1); 
   \draw
   (3,-1) -- (3,-4.5);
   \draw
   (1,-1) to [out=270, in=90] (2,-3.5);
  \draw[line width=4pt, white]
   (2,-1) to [out=270, in=90] (0,-3.5);
  \draw
   (2,-1) to [out=270, in=90] (0,-3.5);
  \draw[line width=4pt, white]
   (0,-1) to [out=270, in=90] (1,-3.5);
  \draw
   (0,-1) to [out=270, in=90] (1,-3.5);
   \draw[red]
   (1,-3.5) -- (1.5,-4.5) -- (2,-3.5);
   \draw
   (0,-3.5) -- (0,-4.5);
  \node at (3.75,-2.25) {$=$};
  \filldraw
  	(3,-4.5) circle (1.5pt)
	(0,-3.5) circle (1.5pt)   
  	(3,-3.5) circle (1.5pt)   
  	(0,-4.5) circle (1.5pt);
  \filldraw[blue]
  	(3,-1) circle (1.5pt)   
  	(2,-1) circle (1.5pt)   
   	(2.5,0) circle (1.5pt)   
   	(2.5,0) circle (1.5pt);
  \filldraw[red]
   	(0.5,0) circle (1.5pt)   
   	(0,-1) circle (1.5pt)  
	(1,-1) circle (1.5pt)   
	(1,-3.5) circle (1.5pt)   
	(2,-3.5) circle (1.5pt)   
	(1.5,-4.5) circle (1.5pt);

 \end{scope}

 \begin{scope}[xshift=9.5cm]
  \draw
   (0.2,0) -- (0.2,-4.5)
   (3,0) -- (3,-1)
   (4,0) -- (4,-3.5)
   (1,-3.5) -- (1,-1);
   \draw[red]
   (1,-1)-- (1.5,0) -- (2,-1);
   \draw
   (3,-1) to [out=270, in=90] (2,-3.5);
  \draw[line width=4pt, white]
   (2,-1) to [out=270, in=90] (3,-3.5);
  \draw
   (2,-1) to [out=270, in=90] (3,-3.5);
   \draw[red]
   (1,-3.5) -- (1.5,-4.5) -- (2,-3.5);
   \draw[blue]
   (3,-3.5) -- (3.5,-4.5) -- (4,-3.5);
  \filldraw
   	(0.2,0) circle (1.5pt)
   	(3,0) circle (1.5pt)
	(4,0) circle (1.5pt)
	(0.2,-1) circle (1.5pt)
	(3,-1) circle (1.5pt)   
	(4,-1) circle (1.5pt)
	(0.2,-3.5) circle (1.5pt)
	(0.2,-4.5) circle (1.5pt);
   \filldraw[red] 
	(1.5,0) circle (1.5pt)
	(1,-1) circle (1.5pt)   
	(2,-1) circle (1.5pt)    
	(1,-3.5) circle (1.5pt)   
	(2,-3.5) circle (1.5pt)   
	(1.5,-4.5) circle (1.5pt);
	\filldraw[blue]
	(3,-3.5) circle (1.5pt)   
	(4,-3.5) circle (1.5pt)   
	(3.5,-4.5) circle (1.5pt);
 \end{scope}
\end{tikzpicture}
\caption{Multiplication of multicolored spraiges.}
\label{fig:spraiges-multiplication}
\end{figure}

The operation $\ast$ defined for $s\Vbr$ can be defined in general for multicolored spraiges, via concatenation of diagrams.  It is only defined for certain pairs of multicolored spraiges, namely we can multiply $\sigma_1\ast\sigma_2$ for $\sigma_1\in s_1\spraige_{n_1,m_1}$ and $\sigma_2 \in s_2\spraige_{n_2,m_2}$ if and only if $m_1=n_2$.  In this case we obtain $\sigma_1 \ast \sigma_2 \in \max\{s_1, s_2\}\spraige_{n_1,m_2}$. For an example of the multiplication of multicolored spraiges see Figure \ref{fig:spraiges-multiplication}.

\begin{remark}
	As a remark, the following observations show that $s\spraige$ is a groupoid.
\begin{itemize}
\item[(i)] For every $n\in \N$ there is an identity multicolored
$(n,n)$-spraige $1_n$ with respect to $\ast$, namely the multicolored spraige represented by $(1_n,\id,1_n)$.  Here, by abuse of notation, $1_n$ also denotes the trivial forest with $n$ roots.

\item[(ii)] For every multicolored $(n,m)$-spraige $(F_-,b,F_+)$ there exists an inverse multicolored $(m,n)$-spraige $(F_+,b^{-1},F_-)$, in the sense that
\begin{align*}
(F_-,b,F_+) \ast (F_+,b^{-1},F_-) & = 1_n \\
\intertext{and}
(F_+,b^{-1},F_-)\ast(F_-,b,F_+) & = 1_m \text{\,.}
\end{align*}

\end{itemize}
\end{remark}

\medskip


\begin{figure}[t]
\centering
\begin{tikzpicture}
 \begin{scope}[yshift=-1cm, line width=1pt,scale=0.9]
 	\coordinate (a) at (0,0);
 	\coordinate (b) at ($(a) + (1,0)$);
 	\coordinate (c) at ($(a) + (2,0)$);
 	\coordinate (d) at ($(a) + (3,0)$);
 	\coordinate (e) at ($(a) + (4,0)$);
 	\coordinate (f) at ($(a) + (5,0)$);
 	\coordinate (g) at ($(a) + (0.5,-1)$);
 	\coordinate (h) at ($(a) + (1.5,-1)$);
 	\coordinate (i) at ($(a) + (3.5,-1)$);
 	\coordinate (j) at ($(a) + (4.5,-1)$);

	\draw[red]
	(g) -- (b) -- (h);
	
	\draw[blue]
	(i) -- (e) -- (j);
 	
	\filldraw
(a) circle (1pt)
(c) circle (1pt)
(d) circle (1pt)
(f) circle (1pt);
\filldraw[red]
(b) circle (1pt)
(g) circle (1pt)
(h) circle (1pt);
\filldraw[blue]
(e) circle (1pt)
(i) circle (1pt)
(j) circle (1pt);
 \end{scope}

\begin{scope}[yshift=1.3cm, line width=1pt,scale=0.9]
	\coordinate (a) at (0,0);
	\coordinate (b) at ($(a) + (1,0)$);
	\coordinate (c) at ($(a) + (2,0)$);
	\coordinate (d) at ($(a) + (3,0)$);
	\coordinate (e) at ($(a) + (4,0)$);
	\coordinate (f) at ($(a) + (5,0)$);
	\coordinate (g) at ($(a) + (0.5,-1)$);
	\coordinate (h) at ($(a) + (1.5,-1)$);
	\coordinate (i) at ($(a) + (3.5,-1)$);
	\coordinate (j) at ($(a) + (4.5,-1)$);
	\coordinate (k) at ($(a) + (4,-2)$);
	\coordinate (l) at ($(a) + (5,-2)$);
	
	\draw[red]
	(g) -- (b) -- (h)
	(k) -- (j) -- (l);
	
	\draw[blue]
	(i) -- (e) -- (j);
	
	\filldraw
	(a) circle (1pt)
	(c) circle (1pt)
	(d) circle (1pt)
	(f) circle (1pt);
	\filldraw[red]
	(b) circle (1pt)
	(g) circle (1pt)
	(h) circle (1pt)
	(k) circle (1pt)
	(l) circle (1pt)
	(j) circle (1pt);
	\filldraw[blue]
	(e) circle (1pt)
	(i) circle (1pt);
\end{scope}

\end{tikzpicture}
\caption{On top is a multicolored elementary forest and on the bottom is a multicolored very elementary forest.}
\label{fig:elem_forests}
\end{figure}
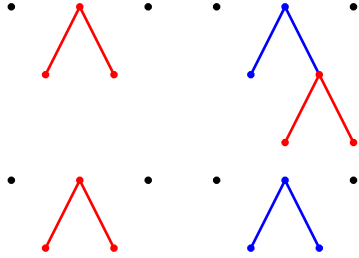

We will need a further understanding of a certain characteristic of a multicolored forest in order to set up the Stein complex in following section. 
\begin{definition}
	A multicolored forest is characterized by being \textit{very elementary} if there is at most one caret per head, and any such forest is then called a \textit{very elementary multicolored forest}. A multicolored forest is called \textit{elementary} if there is at most one caret per color per head, and any such forest is then called an \textit{elementary multicolored forest}. See Figure \ref{fig:elem_forests} for an example of an elementary multicolored forest and a very elementary multicolored forest.
\end{definition}

Observe that if $s=1$, that is, we are only dealing with a single color for the entire tree, then elementary and very elementary mean the same thing. 

For a given multicolored $(n, m)$-spraige $\sigma$, consider any multicolored forest $F$ with $m$ roots and $l$ leaves. Define a \textit{multicolored splitting of $\sigma$ by $F$} as multiplying $\sigma$ on the right by the multicolored spraige $(F, id, 1_l)$. In a similar manner, one says that we have a \textit{multicolored merging of $\sigma$ by $F'$} when we multiply on the right by the multicolored spraige $(1_m, id, F')$.  Moreover, if $F$ (respectively $F'$) is a (very) elementary forest, then we call this process a \textit{(very) elementary splitting} (respectively \textit{(very) elementary merging}). If $F$ (respectively $F'$) were a single caret we will shorten it by saying \textit{adding a split} (respectively \textit{adding a merge}) to the multicolored spraige. For an illustration of this splitting of a multicolored spraige see Figure \ref{fig:multicolored_spraige_splitting} and this multicolored merging of a spraige see Figure \ref{fig:multicolored_spraige_merging}.

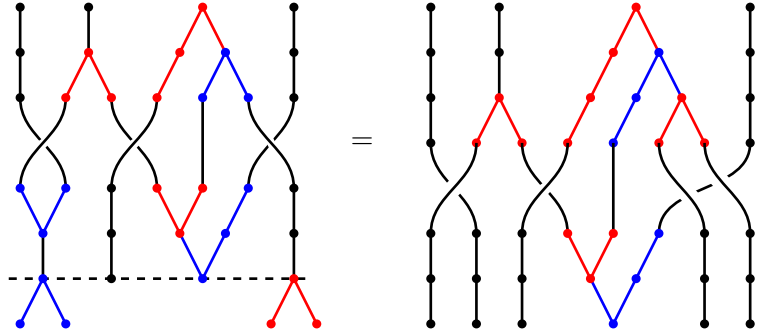
\begin{figure}[t]
	\centering
		\begin{tikzpicture}[line width=1pt,scale=0.6]
		\coordinate (a) at (0,0);
		\coordinate (b) at ($(a)+(1,0)$);
		\coordinate (c) at ($(a)+(2,0)$);
		\coordinate (d) at ($(a)+(3,0)$);
		\coordinate (e) at ($(a)+(4,0)$);
		\coordinate (f) at ($(a)+(5,0)$);
		\coordinate (g) at ($(a)+(6,0)$);
		\coordinate (h) at ($(a)+(0,1)$);
		\coordinate (i) at ($(a)+(1.5,1)$);
		\coordinate (j) at ($(a)+(3.5,1)$);
		\coordinate (k) at ($(a)+(4.5,1)$);
		\coordinate (l) at ($(a)+(4,2)$);
		\coordinate (m) at ($(a)+(0,2)$);
		\coordinate (na) at ($(a)+(1.5,2)$);
		\coordinate (n) at ($(a)+(6,1)$);
		\coordinate (o) at ($(a)+(6,2)$);
		
		\draw
		(a) -- (m)
		(g) -- (o)
		(i) -- (na);
		\draw[red]
		(b) -- (i) -- (c)
		(d) -- (l) -- (k);
		\draw[blue]
		(e) -- (k) -- (f);

		\coordinate (p) at ($(a)+(0,-2)$);
		\coordinate (q) at ($(a)+(1,-2)$);
		\coordinate (r) at ($(a)+(2,-2)$);
		\coordinate (s) at ($(a)+(3,-2)$);
		\coordinate (t) at ($(a)+(4,-2)$);
		\coordinate (u) at ($(a)+(5,-2)$);
		\coordinate (v) at ($(a)+(6,-2)$);
		\coordinate (w) at ($(a)+(.5,-3)$);
		\coordinate (wa) at ($(a)+(.5,-4)$);
		\coordinate (x) at ($(a)+(2,-3)$);
		\coordinate (xa) at ($(a)+(2,-4)$);
		\coordinate (y) at ($(a)+(3.5,-3)$);
		\coordinate (z) at ($(a)+(4.5,-3)$);
		\coordinate (aa) at ($(a)+(6,-3)$);
		\coordinate (ab) at ($(a)+(4,-4)$);
		\coordinate (ac) at ($(a)+(6,-4)$);

		\draw
		 (r) -- (x) -- (xa)
		 (v) -- (aa) -- (ac)
		 (w) -- (wa);
		\draw[blue]
		 (p) -- (w) -- (q)
		 (y) -- (ab) -- (z) -- (u);
		\draw[red]
		 (s) -- (y) -- (t);
		 
		\draw
		 (a) to [out=-90, in=90] (q)
		 (c) to [out=-90, in=90] (s)
		 (g) to [out=-90, in=90] (u);
 		\draw[white, line width=3pt]
 		 (1.25, 0) to [out=-90, in=90] (-.25, -2)   
		 (.75,-0.5) to [out=-90, in=90] (-.25,-2)
		 (3.25, 0) to [out=-90, in=90] (1.75, -2)
		 (2.75,-0.5) to [out=-90, in=90] (1.75,-2)
		 (5.25,-0.1) to [out=-90, in=90] (5.75,-1.9);
		\draw
	     (b) to [out=-90, in=90] (p)
	     (d) to [out=-90, in=90] (r)
	     (f) to [out=-90, in=90] (v); 
		 \draw
		 (e) to [out=-90, in=90] (t);   
		 
		 \draw[dashed]
		 (-.25,-4) -- (6.25, -4);

		 \filldraw (a) circle (2pt);
		 \filldraw[red] (b) circle (2pt);
		 \filldraw[red] (c) circle (2pt);
		 \filldraw[red] (d) circle (2pt);
		 \filldraw[blue] (e) circle (2pt);
		 \filldraw[blue] (f) circle (2pt);
		 \filldraw (g) circle (2pt);
		 \filldraw (h) circle (2pt);
		 \filldraw[red] (i) circle (2pt);
		 \filldraw[red] (j) circle (2pt);
		 \filldraw[blue] (k) circle (2pt);
		 \filldraw[red] (l) circle (2pt);
		 \filldraw (m) circle (2pt);
		 \filldraw (na) circle (2pt);
		 \filldraw (n) circle (2pt);
		 \filldraw (o) circle (2pt);
		 
		 \filldraw[blue] (p) circle (2pt);
		 \filldraw[blue] (q) circle (2pt);
		 \filldraw (r) circle (2pt);
		 \filldraw[red] (s) circle (2pt);
		 \filldraw[red] (t) circle (2pt);
		 \filldraw[blue] (u) circle (2pt);
		 \filldraw (v) circle (2pt);
		 \filldraw[blue] (w) circle (2pt);
		 \filldraw[blue] (wa) circle (2pt);
		 \filldraw (x) circle (2pt);
		 \filldraw (xa) circle (2pt);
		 \filldraw[red] (y) circle (2pt);
		 \filldraw[blue] (z) circle (2pt);
		 \filldraw (aa) circle (2pt);
		 \filldraw[blue] (ab) circle (2pt);
		 \filldraw[red] (ac) circle (2pt);
		 
		\coordinate (wb) at ($(wa)+(-.5,-1)$);
		\coordinate (wc) at ($(wa)+(.5,-1)$);
		\coordinate (ad) at ($(ac)+(-.5,-1)$);
		\coordinate (ae) at ($(ac)+(.5,-1)$);
		
		\filldraw[blue] (wb) circle (2pt);
		\filldraw[blue] (wc) circle (2pt);
		\filldraw[red] (ad) circle (2pt);
		\filldraw[red] (ae) circle (2pt);
		\draw[blue]
		 (wb) -- (wa) -- (wc);
		\draw[red]
		 (ad) -- (ac) -- (ae);
		
		\node at (7.5,-1) {$=$};
		
		\coordinate (a') at (9,-1);
		\coordinate (b') at ($(a')+(1,0)$);
		\coordinate (c') at ($(a')+(2,0)$);
		\coordinate (d') at ($(a')+(3,0)$);
		\coordinate (e') at ($(a')+(4,0)$);
		\coordinate (f') at ($(a')+(5,0)$);
		\coordinate (g') at ($(a')+(6,0)$);
		\coordinate (h') at ($(a')+(7,0)$);
		\coordinate (i') at ($(a')+(0,1)$);
		\coordinate (j') at ($(a')+(1.5,1)$);
		\coordinate (k') at ($(a')+(3.5,1)$);
		\coordinate (l') at ($(a')+(4.5,1)$);
		\coordinate (m') at ($(a')+(5.5,1)$);
		\coordinate (n') at ($(a')+(7,1)$);
		\coordinate (o') at ($(a')+(0,2)$);
		\coordinate (p') at ($(a')+(1.5,2)$);
		\coordinate (q') at ($(a')+(4,2)$);
		\coordinate (r') at ($(a')+(5,2)$);
		\coordinate (s') at ($(a')+(7,2)$);
		\coordinate (t') at ($(a')+(0,3)$);
		\coordinate (u') at ($(a')+(1.5,3)$);
		\coordinate (v') at ($(a')+(4.5,3)$);
		\coordinate (w') at ($(a')+(7,3)$);
		
		\draw
		(a') -- (t')
		(j') -- (u')
		(h') -- (w');
		\draw[blue]
		(e') -- (r') -- (m');
		\draw[red]
		(b') -- (j') -- (c')
		(d') -- (v') -- (r')
		(f') -- (m') -- (g');
		
		\filldraw (a') circle (2pt);
		\filldraw[red] (b') circle (2pt);
		\filldraw[red] (c') circle (2pt);
		\filldraw[red] (d') circle (2pt);
		\filldraw[blue] (e') circle (2pt);
		\filldraw[red] (f') circle (2pt);
		\filldraw[red] (g') circle (2pt);
		\filldraw (h') circle (2pt);
		\filldraw (i') circle (2pt);
		\filldraw[red] (j') circle (2pt);
		\filldraw[red] (k') circle (2pt);
		\filldraw[blue] (l') circle (2pt);
		\filldraw[red] (m') circle (2pt);
		\filldraw (n') circle (2pt);
		\filldraw (o') circle (2pt);
		\filldraw (p') circle (2pt);
		\filldraw[red] (q') circle (2pt);
		\filldraw[blue] (r') circle (2pt);
		\filldraw (s') circle (2pt);
		\filldraw (t') circle (2pt);
		\filldraw (u') circle (2pt);
		\filldraw[red] (v') circle (2pt);
		\filldraw (w') circle (2pt);
		
		\coordinate (x') at ($(a')+(0,-2)$);
		\coordinate (y') at ($(a')+(1,-2)$);
		\coordinate (z') at ($(a')+(2,-2)$);
		\coordinate (aa') at ($(a')+(3,-2)$);
		\coordinate (ab') at ($(a')+(4,-2)$);
		\coordinate (ac') at ($(a')+(5,-2)$);
		\coordinate (ad') at ($(a')+(6,-2)$);
		\coordinate (ae') at ($(a')+(7,-2)$);
		\coordinate (af') at ($(a')+(0,-3)$);
		\coordinate (ag') at ($(a')+(1,-3)$);
		\coordinate (ah') at ($(a')+(2,-3)$);
		\coordinate (ai') at ($(a')+(3.5,-3)$);
		\coordinate (aj') at ($(a')+(4.5,-3)$);
		\coordinate (ak') at ($(a')+(6,-3)$);
		\coordinate (al') at ($(a')+(7,-3)$);
		\coordinate (am') at ($(a')+(0,-4)$);
		\coordinate (an') at ($(a')+(1,-4)$);
		\coordinate (ao') at ($(a')+(2,-4)$);
		\coordinate (ap') at ($(a')+(4,-4)$);
		\coordinate (aq') at ($(a')+(6,-4)$);
		\coordinate (ar') at ($(a')+(7,-4)$);
		
		\draw
		(x') -- (am')
		(y') -- (an')
		(z') -- (ao')
		(ad') -- (aq')
		(ae') -- (ar');
		\draw[blue]
		(ai') -- (ap') -- (ac');
		\draw[red]
		(aa') -- (ai') -- (ab');
		
				\draw
				(a') to [out=-90, in=90] (y')
				(c') to [out=-90, in=90] (aa')
				(e') to [out=-90, in=90] (ab')
				(h') to [out=-90, in=90] (ac');
				\draw[white, line width=4pt]
				(9.9, -1.1) to [out=-90, in=90] (8.9, -2.9)   
				(10.1,-1.1) to [out=-90, in=90] (9.1,-2.9)
				(12.25, -1) to [out=-90, in=90] (10.75, -3)
				(11.75,-1.5) to [out=-90, in=90] (10.75,-3)
				(13.9,-1.1) to [out=-90, in=90] (14.9,-2.9)
				(14.1,-1.1) to [out=-90, in=90] (15.1,-2.9)
				(14.9,-1.1) to [out=-90, in=90] (15.9,-2.9)
				(15.1,-1.1) to [out=-90, in=90] (16.1,-2.9);
				\draw
				(b') to [out=-90, in=90] (x')
				(d') to [out=-90, in=90] (z')
				(f') to [out=-90, in=90] (ad')
				(g') to [out=-90, in=90] (ae'); 
				
				\filldraw (x') circle (2pt);
				\filldraw (y') circle (2pt);
				\filldraw (z') circle (2pt);
				\filldraw[red] (aa') circle (2pt);
				\filldraw[red] (ab') circle (2pt);
				\filldraw[blue] (ac') circle (2pt);
				\filldraw (ad') circle (2pt);
				\filldraw (ae') circle (2pt);
				\filldraw (af') circle (2pt);
				\filldraw (ag') circle (2pt);
				\filldraw (ah') circle (2pt);
				\filldraw[red] (ai') circle (2pt);
				\filldraw[blue] (aj') circle (2pt);
				\filldraw (ak') circle (2pt);
				\filldraw (al') circle (2pt);
				\filldraw (am') circle (2pt); 
				\filldraw (an') circle (2pt);
				\filldraw (ao') circle (2pt); 
				\filldraw[blue] (ap') circle (2pt); 
				\filldraw (aq') circle (2pt); 
				\filldraw (ar') circle (2pt); 
				
		\end{tikzpicture}
	\caption{A splitting of a multicolored spraige.}
	\label{fig:multicolored_spraige_splitting}
\end{figure}

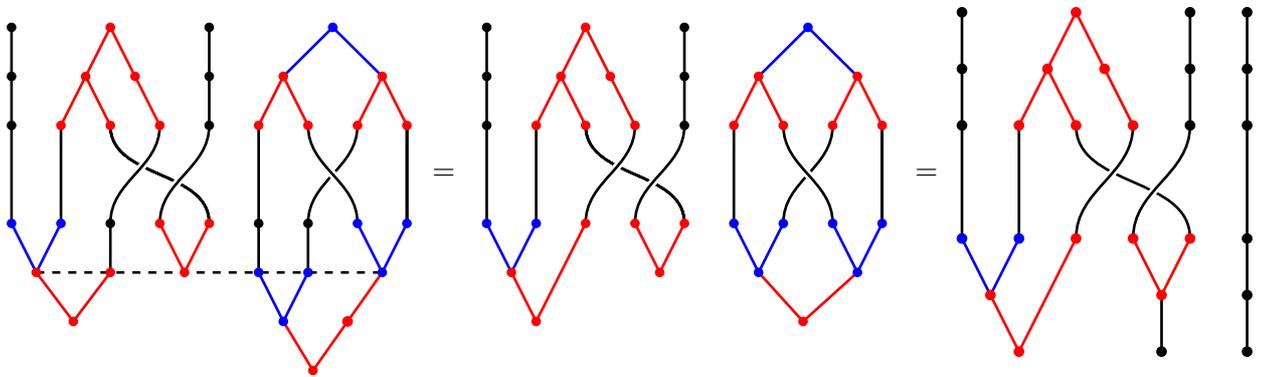
\begin{figure}[t]
	\centering
		\begin{tikzpicture}[line width=1pt]
		\begin{scope}[xshift=-14.25cm,xscale=0.65,yscale=0.65]
			\coordinate (a) at (0,0);
			\coordinate (b) at ($(a)+(1,0)$);
			\coordinate (c) at ($(a)+(2,0)$);
			\coordinate (d) at ($(a)+(3,0)$);
			\coordinate (e) at ($(a)+(4,0)$);
			\coordinate (f) at ($(a)+(5,0)$);
			\coordinate (g) at ($(a)+(6,0)$);
			\coordinate (h) at ($(a)+(7,0)$);
			\coordinate (h') at ($(a)+(8,0)$);
			\coordinate (i) at ($(a)+(0,1)$);
			\coordinate (j) at ($(a)+(1.5,1)$);
			\coordinate (k) at ($(a)+(2.5,1)$);
			\coordinate (l) at ($(a)+(4,1)$);
			\coordinate (m) at ($(a)+(5.5,1)$);
			\coordinate (n) at ($(a)+(7.5,1)$);
			\coordinate (o) at ($(a)+(0,2)$);
			\coordinate (p) at ($(a)+(2,2)$);
			\coordinate (q) at ($(a)+(4,2)$);
			\coordinate (r) at ($(a)+(6.5,2)$);
			
			\coordinate (s) at ($(a)+(0,-2)$);
			\coordinate (t) at ($(a)+(1,-2)$);
			\coordinate (u) at ($(a)+(2,-2)$);
			\coordinate (v) at ($(a)+(3,-2)$);
			\coordinate (w) at ($(a)+(4,-2)$);
			\coordinate (x) at ($(a)+(5,-2)$);
			\coordinate (y) at ($(a)+(6,-2)$);
			\coordinate (z) at ($(a)+(7,-2)$);
			\coordinate (z') at ($(a)+(8,-2)$);
			\coordinate (aa) at ($(a)+(.5,-3)$);
			\coordinate (ab) at ($(a)+(2,-3)$);
			\coordinate (ac) at ($(a)+(3.5,-3)$);
			\coordinate (ad) at ($(a)+(5,-3)$);
			\coordinate (ad') at ($(a)+(6,-3)$);
			\coordinate (ae) at ($(a)+(7.5,-3)$);
			\coordinate (af) at ($(a)+(1.25,-4)$);
			\coordinate (ag) at ($(a)+(5.5,-4)$);
			\coordinate (ah) at ($(a)+(6.8,-4)$);
			\coordinate (ai) at ($(a)+(6.1,-5)$);
			
			\draw
			(a) -- (o)
			(a) -- (s)
			(b) -- (t)
			(e) -- (q)
			(f) -- (x)
			(h') -- (z')
			(u) -- (ab)
			(x) -- (ad)
			(y) -- (ad')
			(h') -- (z');
			\draw[red]
			(b) -- (j) -- (c)
			(f) -- (m) -- (g)
			(j) -- (p) -- (d)
			(h) -- (n) -- (h')
			(v) -- (ac) -- (w)
			(aa) -- (af) -- (ab)
			(ag) -- (ai) -- (ae);
			\draw[blue]
			(m) -- (r) -- (n)
			(s) -- (aa) -- (t)
			(z) -- (ae) -- (z')
			(ad) -- (ag) -- (ad');

			\draw
			(c) to [out=-90, in=90] (w)
			(c) to [out=-90, in=90] (w)
			(h) to [out=-90, in=90] (y);
			\draw[white, line width=3pt]
			(d) to [out=-90, in=90] (u)
			(e) to [out=-90, in=90] (v)
			(g) to [out=-90, in=90] (z);
			\draw
			(d) to [out=-90, in=90] (u)
			(e) to [out=-90, in=90] (v)
			(g) to [out=-90, in=90] (z);
			\draw[dashed]
			(aa) -- (ae);
			
			\node at (8.75,-1) {$=$};
			
			\filldraw (a) circle (2pt);
			\filldraw[red] (b) circle (2pt);
			\filldraw[red] (c) circle (2pt);
			\filldraw[red] (d) circle (2pt);
			\filldraw (e) circle (2pt);
			\filldraw[red] (f) circle (2pt);
			\filldraw[red] (g) circle (2pt);
			\filldraw[red] (h) circle (2pt);
			\filldraw[red] (h') circle (2pt);
			\filldraw (i) circle (2pt);
			\filldraw[red] (j) circle (2pt);
			\filldraw[red] (k) circle (2pt);
			\filldraw (l) circle (2pt);
			\filldraw[red] (m) circle (2pt);
			\filldraw[red] (n) circle (2pt);
			\filldraw (o) circle (2pt);
			\filldraw[red] (p) circle (2pt);
			\filldraw (q) circle (2pt);
			\filldraw[blue] (r) circle (2pt);
			
			\filldraw[blue] (s) circle (2pt);
			\filldraw[blue] (t) circle (2pt);
			\filldraw (u) circle (2pt);
			\filldraw[red] (v) circle (2pt);
			\filldraw[red] (w) circle (2pt);
			\filldraw (x) circle (2pt);
			\filldraw (y) circle (2pt);
			\filldraw[blue] (z) circle (2pt);
			\filldraw[blue] (z') circle (2pt);
			\filldraw[red] (aa) circle (2pt);
			\filldraw[red] (ab) circle (2pt);
			\filldraw[red] (ac) circle (2pt);
			\filldraw[blue] (ad) circle (2pt);
			\filldraw[blue] (ad') circle (2pt);
			\filldraw[blue] (ae) circle (2pt);
			\filldraw[red] (af) circle (2pt);
			\filldraw[blue] (ag) circle (2pt);
			\filldraw[red] (ah) circle (2pt);
			\filldraw[red] (ai) circle (2pt);			

		\end{scope}

		\begin{scope}[xshift=-8cm,xscale=0.65,yscale=0.65]
	\coordinate (a) at (0,0);
	\coordinate (b) at ($(a)+(1,0)$);
	\coordinate (c) at ($(a)+(2,0)$);
	\coordinate (d) at ($(a)+(3,0)$);
	\coordinate (e) at ($(a)+(4,0)$);
	\coordinate (f) at ($(a)+(5,0)$);
	\coordinate (g) at ($(a)+(6,0)$);
	\coordinate (h) at ($(a)+(7,0)$);
	\coordinate (h') at ($(a)+(8,0)$);
	\coordinate (i) at ($(a)+(0,1)$);
	\coordinate (j) at ($(a)+(1.5,1)$);
	\coordinate (k) at ($(a)+(2.5,1)$);
	\coordinate (l) at ($(a)+(4,1)$);
	\coordinate (m) at ($(a)+(5.5,1)$);
	\coordinate (n) at ($(a)+(7.5,1)$);
	\coordinate (o) at ($(a)+(0,2)$);
	\coordinate (p) at ($(a)+(2,2)$);
	\coordinate (q) at ($(a)+(4,2)$);
	\coordinate (r) at ($(a)+(6.5,2)$);
	
	\coordinate (s) at ($(a)+(0,-2)$);
	\coordinate (t) at ($(a)+(1,-2)$);
	\coordinate (u) at ($(a)+(2,-2)$);
	\coordinate (v) at ($(a)+(3,-2)$);
	\coordinate (w) at ($(a)+(4,-2)$);
	\coordinate (x) at ($(a)+(5,-2)$);
	\coordinate (y) at ($(a)+(6,-2)$);
	\coordinate (z) at ($(a)+(7,-2)$);
	\coordinate (z') at ($(a)+(8,-2)$);
	\coordinate (aa) at ($(a)+(.5,-3)$);
	\coordinate (ab) at ($(a)+(2,-3)$);
	\coordinate (ac) at ($(a)+(3.5,-3)$);
	\coordinate (ad) at ($(a)+(5,-3)$);
	\coordinate (ad') at ($(a)+(6,-3)$);
	\coordinate (ae) at ($(a)+(7.5,-3)$);
	\coordinate (af) at ($(a)+(1,-4)$);
	\coordinate (ag) at ($(a)+(5.5,-3)$);
	\coordinate (ai) at ($(a)+(6.4,-4)$);
	\coordinate (aj) at ($(a)+(6.4,-4)$);
	
	\draw
	(a) -- (o)
	(a) -- (s)
	(b) -- (t)
	(e) -- (q)
	(f) -- (x)
	(h') -- (z')
	(h') -- (z');
	\draw[red]
	(b) -- (j) -- (c)
	(f) -- (m) -- (g)
	(j) -- (p) -- (d)
	(h) -- (n) -- (h')
	(v) -- (ac) -- (w)
	(aa) -- (af) -- (u)
	(ag) -- (ai) -- (ae);
	\draw[blue]
	(m) -- (r) -- (n)
	(s) -- (aa) -- (t)
	(z) -- (ae) -- (z')
	(x) -- (ag) -- (y);

	\draw
	(c) to [out=-90, in=90] (w)
	(c) to [out=-90, in=90] (w)
	(h) to [out=-90, in=90] (y);
	\draw[white, line width=3pt]
	(d) to [out=-90, in=90] (u)
	(e) to [out=-90, in=90] (v)
	(g) to [out=-90, in=90] (z);
	\draw
	(d) to [out=-90, in=90] (u)
	(e) to [out=-90, in=90] (v)
	(g) to [out=-90, in=90] (z);
	
	\node at (8.9,-1) {$=$};
	
	\filldraw (a) circle (2pt);
	\filldraw[red] (b) circle (2pt);
	\filldraw[red] (c) circle (2pt);
	\filldraw[red] (d) circle (2pt);
	\filldraw (e) circle (2pt);
	\filldraw[red] (f) circle (2pt);
	\filldraw[red] (g) circle (2pt);
	\filldraw[red] (h) circle (2pt);
	\filldraw[red] (h') circle (2pt);
	\filldraw (i) circle (2pt);
	\filldraw[red] (j) circle (2pt);
	\filldraw[red] (k) circle (2pt);
	\filldraw (l) circle (2pt);
	\filldraw[red] (m) circle (2pt);
	\filldraw[red] (n) circle (2pt);
	\filldraw (o) circle (2pt);
	\filldraw[red] (p) circle (2pt);
	\filldraw (q) circle (2pt);
	\filldraw[blue] (r) circle (2pt);
	
	\filldraw[blue] (s) circle (2pt);
	\filldraw[blue] (t) circle (2pt);
	\filldraw[red] (u) circle (2pt);
	\filldraw[red] (v) circle (2pt);
	\filldraw[red] (w) circle (2pt);
	\filldraw[blue] (x) circle (2pt);
	\filldraw[blue] (y) circle (2pt);
	\filldraw[blue] (z) circle (2pt);
	\filldraw[blue] (z') circle (2pt);
	\filldraw[red] (aa) circle (2pt);
	\filldraw[red] (ac) circle (2pt);
	\filldraw[blue] (ae) circle (2pt);
	\filldraw[red] (af) circle (2pt);
	\filldraw[blue] (ag) circle (2pt);
	\filldraw[red] (ah) circle (2pt);
	\filldraw[red] (ai) circle (2pt);
		\end{scope}
	
		\begin{scope}[xshift=-1.75cm,xscale=0.75,yscale=0.75]
\coordinate (a) at (0,0);
\coordinate (b) at ($(a)+(1,0)$);
\coordinate (c) at ($(a)+(2,0)$);
\coordinate (d) at ($(a)+(3,0)$);
\coordinate (e) at ($(a)+(4,0)$);
\coordinate (f) at ($(a)+(5,0)$);
\coordinate (f') at ($(a)+(5,1)$);
\coordinate (g) at ($(a)+(6,0)$);
\coordinate (h) at ($(a)+(7,0)$);
\coordinate (h') at ($(a)+(8,0)$);
\coordinate (i) at ($(a)+(0,1)$);
\coordinate (j) at ($(a)+(1.5,1)$);
\coordinate (k) at ($(a)+(2.5,1)$);
\coordinate (l) at ($(a)+(4,1)$);
\coordinate (m) at ($(a)+(5.5,1)$);
\coordinate (n) at ($(a)+(7.5,1)$);
\coordinate (o) at ($(a)+(0,2)$);
\coordinate (p) at ($(a)+(2,2)$);
\coordinate (q) at ($(a)+(4,2)$);
\coordinate (r) at ($(a)+(5,2)$);

\coordinate (s) at ($(a)+(0,-2)$);
\coordinate (t) at ($(a)+(1,-2)$);
\coordinate (u) at ($(a)+(2,-2)$);
\coordinate (v) at ($(a)+(3,-2)$);
\coordinate (w) at ($(a)+(4,-2)$);
\coordinate (x) at ($(a)+(5,-2)$);
\coordinate (x') at ($(a)+(5,-3)$);
\coordinate (x'') at ($(a)+(5,-4)$);
\coordinate (y) at ($(a)+(6,-2)$);
\coordinate (z) at ($(a)+(7,-2)$);
\coordinate (z') at ($(a)+(8,-2)$);
\coordinate (aa) at ($(a)+(.5,-3)$);
\coordinate (ab) at ($(a)+(2,-3)$);
\coordinate (ac) at ($(a)+(3.5,-3)$);
\coordinate (ad) at ($(a)+(5,-3)$);
\coordinate (ad') at ($(a)+(6,-3)$);
\coordinate (ae) at ($(a)+(7.5,-3)$);
\coordinate (af) at ($(a)+(1,-4)$);
\coordinate (ag) at ($(a)+(5.5,-3)$);
\coordinate (ai) at ($(a)+(6.4,-4)$);
\coordinate (aj) at ($(a)+(3.5,-4)$);

\draw
(a) -- (o)
(a) -- (s)
(b) -- (t)
(e) -- (q)
(r) -- (f) -- (x) -- (x') -- (x'')
(ac) -- (aj);
\draw[red]
(b) -- (j) -- (c)
(j) -- (p) -- (d)
(v) -- (ac) -- (w)
(aa) -- (af) -- (u);
\draw[blue]
(s) -- (aa) -- (t);

\draw
(c) to [out=-90, in=90] (w);
\draw[white, line width=3pt]
(d) to [out=-90, in=90] (u)
(e) to [out=-90, in=90] (v);
\draw
(d) to [out=-90, in=90] (u)
(e) to [out=-90, in=90] (v);

\filldraw (a) circle (2pt);
\filldraw[red] (b) circle (2pt);
\filldraw[red] (c) circle (2pt);
\filldraw[red] (d) circle (2pt);
\filldraw (e) circle (2pt);
\filldraw (f) circle (2pt);
\filldraw (f') circle (2pt);
\filldraw (i) circle (2pt);
\filldraw[red] (j) circle (2pt);
\filldraw[red] (k) circle (2pt);
\filldraw (l) circle (2pt);
\filldraw (o) circle (2pt);
\filldraw[red] (p) circle (2pt);
\filldraw (q) circle (2pt);
\filldraw (r) circle (2pt);

\filldraw[blue] (s) circle (2pt);
\filldraw[blue] (t) circle (2pt);
\filldraw[red] (u) circle (2pt);
\filldraw[red] (v) circle (2pt);
\filldraw[red] (w) circle (2pt);
\filldraw (x) circle (2pt);
\filldraw (x') circle (2pt);
\filldraw (x'') circle (2pt);
\filldraw[red] (aa) circle (2pt);
\filldraw[red] (ac) circle (2pt);
\filldraw[red] (af) circle (2pt);
\filldraw[red] (ah) circle (2pt);
\filldraw (aj) circle (2pt);
\end{scope}

		\end{tikzpicture}
	\caption{From left to right: An elementary merging of a multicolored spraige, the resulting multicolored spraige, and an equivalent spraige after applying the braided cross relation.}
	\label{fig:multicolored_spraige_merging}
\end{figure}

We can specialize multicolored spraiges to multicolored braiges as follows. 

\begin{definition}
	We call a  multicolored spraige a \textit{multicolored braid-merge diagram}, abbreviated \textit{multicolored braige} if there are no splits, that is, it is of the form $(1_n, b, F) $ for $b \in B_n$ and $F$ is a multicolored forest having $n$ leaves. Moreover, if $F$ is a (very) elementary forest then we have an \textit{(very) elementary braige}. Analogously to multicolored spraiges, we say \textit{multicolored $n$-braige} and \textit{(very) elementary multicolored $n$-braige} to mean a multicolored braige with $n$ heads and any number of feet with the latter having a (very) elementary multicolored forest as defined.
\end{definition}

\subsection{Dangling multicolored spraiges}
\label{sec:dangling}

We can identify the braid group $B_n$ with a subgroup of $s\spraige_{n,n}$ via $b\mapsto (1_n,b,1_n)$. In particular for any $n,m\in\N$ there is a right action of the braid group $B_m$ on $2\spraige_{n,m}$, by right multiplication. Quotienting out modulo this action encodes the idea that the feet of a multicolored spraige may ``dangle''.  See Figure \ref{fig:dangle} for an example of the dangling action of $B_2$ on $2\spraige_{4,2}$.

\begin{definition}
	For $\sigma \in s\spraige_{n,m}$, denote the orbit of $\sigma$ under the quotient of this action by $[\sigma]$ and call the orbit a \textit{dangling multicolored $(n, m)$-spraige}. We mean \textit{dangling multicolored $n$-spraige} and \textit{dangling multicolored spraige} in the same manner as with other types of multicolored spraiges. The collection of all dangling multicolored spraiges will be denoted by $s\Poset$, with $s\Poset_{n,m}$ denoting the collection of all dangling multicolored $(n, m)$-spraiges and $s\Poset_n$ denoting the collection of all dangling multicolored $n$-spraiges. 
\end{definition}

Moreover, the action of $B_m$ preserves the property of being a multicolored braige or (very) elementary multicolored braige, so we refer to dangling multicolored braiges and dangling (very) elementary multicolored braiges in similar contexts. When $m=1$, $B_m$ is trivial, so we will identify $s\Poset_{n,1}$ with $s\spraige_{n,1}$ for each $n$. We denote the collection of dangling elementary multicolored $n$-braiges by $s\elbraigecpx_n$ and the collection of dangling very elementary multicolored $n$-braiges by $s\velbraigecpx_n$.

\begin{figure}[t]
	\centering
	\begin{tikzpicture}[line width=1pt]
		
		\draw
		(0,-1.5) to [out=270, in=90] (-1,-4.5)
		(1,-1.5) to [out=270, in=90] (0,-4.5)
		(2,-1.5) -- (2,-4.5);
		\draw [line width=4pt, white]
		(-1,-1.5) to [out=270, in=90] (1,-4.5);
		\draw
		(-1,-1.5) to [out=270, in=90] (1,-4.5);
		\draw[red]
		(-1,-4.5) -- (-0.5,-5) -- (0,-4.5);
		\draw[blue]
		(1,-4.5) -- (1.5,-5) -- (2,-4.5);
		\filldraw
		(-1,-1.5) circle (1.5pt)
		(0,-1.5) circle (1.5pt)
		(1,-1.5) circle (1.5pt)  
		(2,-1.5) circle (1.5pt);
		\filldraw[blue]
		(1,-4.5) circle (1.5pt)  
		(2,-4.5) circle (1.5pt)
		(1.5,-5) circle (1.5pt);
		\filldraw[red]
		(-1,-4.5) circle (1.5pt)   
		(0,-4.5) circle (1.5pt)
		(-0.5,-5) circle (1.5pt);
		
		\begin{scope}[xshift=-3.5cm,xscale=-1]
			\draw
			(0,-1.5) to [out=270, in=90] (-1,-4.5)
			(1,-1.5) to [out=270, in=90] (0,-4.5)
			(2,-1.5) -- (2,-4.5);
			\draw [line width=4pt, white]
			(-1,-1.5) to [out=270, in=90] (1,-4.5);
			\draw
			(-1,-1.5) to [out=270, in=90] (1,-4.5);
			\draw[red]
			(-1,-4.5) -- (-0.5,-5) -- (0,-4.5);
			\draw[blue]
			(1,-4.5) -- (1.5,-5) -- (2,-4.5);
			\filldraw
			(-1,-1.5) circle (1.5pt)   
			(0,-1.5) circle (1.5pt)
			(1,-1.5) circle (1.5pt)  
			(2,-1.5) circle (1.5pt);
			\filldraw[red]
			(-1,-4.5) circle (1.5pt)   
			(0,-4.5) circle (1.5pt)
			(-0.5,-5) circle (1.5pt);
			\filldraw[blue]
			(1,-4.5) circle (1.5pt)  
			(2,-4.5) circle (1.5pt)
			(1.5,-5) circle (1.5pt);
		\end{scope}
		
		\begin{scope}[xshift=4.5cm]
			\draw
			(0,-1.5) to [out=270, in=90] (-1,-3)
			(1,-1.5) to [out=270, in=90] (0,-3)
			(2,-1.5) -- (2,-3)
			(1,-3) to [out=270, in=90] (-1,-4.5)
			(2,-3) to [out=270, in=90] (0,-4.5);
			\draw [line width=4pt, white]
			(-1,-1.5) to [out=270, in=90] (1,-3)
			(0,-3) to [out=270, in=90] (2,-4.5)
			(-1,-3) to [out=270, in=90] (1,-4.5);
			\draw
			(-1,-1.5) to [out=270, in=90] (1,-3)
			(0,-3) to [out=270, in=90] (2,-4.5)
			(-1,-3) to [out=270, in=90] (1,-4.5);
			\draw[blue]
			(-1,-4.5) -- (-0.5,-5) -- (0,-4.5);
			\draw[red]
			(1,-4.5) -- (1.5,-5) -- (2,-4.5);
			\filldraw
			(-1,-1.5) circle (1.5pt)   
			(0,-1.5) circle (1.5pt)   
			(1,-1.5) circle (1.5pt)  
			(2,-1.5) circle (1.5pt);
			\filldraw[blue]
			(-1,-4.5) circle (1.5pt)   
			(0,-4.5) circle (1.5pt)   
			(-0.5,-5) circle (1.5pt);
			\filldraw[red]
			(1,-4.5) circle (1.5pt)  
			(2,-4.5) circle (1.5pt)
			(1.5,-5) circle (1.5pt);
		\end{scope}
	\end{tikzpicture}
	\caption{Multicolored dangling.}
	\label{fig:dangle}
\end{figure}
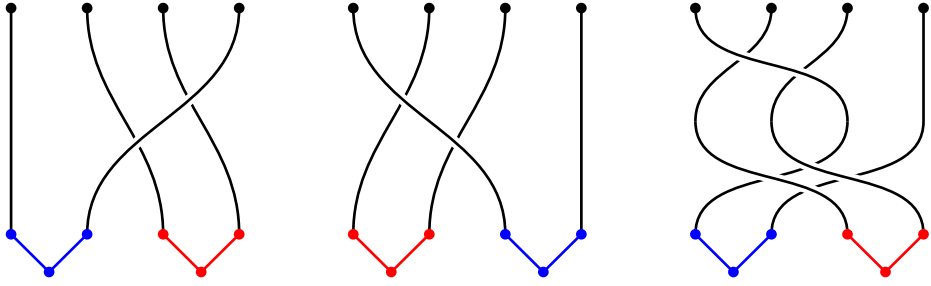

There exists a poset structure on $s\Poset$. Let $x,y \in s\Poset$ with $x = [\sigma_x]$. We say that $x \leq y$ if there exists a multicolored forest $F$ with $m$ leaves such that $y = [\sigma_x \ast (F, id, 1_m)]$. In other words, $x \leq y$ if $y$ is obtained from $x$ via multicolored splittings. Observe that each element can be obtained from itself via the trivial splitting. Moreover, if $x \leq y$ and $y \leq x$, i.e., $x$ and $y$ can be obtained from each other via multicolored splittings, then these splittings must have been trivial, hence $x = y$. Finally, if $x \leq y$ and $y \leq z$ in $s\Poset$, then one can get from $x$ to $y$ and from $y$ to $z$ via splittings, hence we can get from $x$ to $z$ via the concatenation of these splittings so $x \leq z$. Therefore, $\leq$ is a partial order. Also, if $x \in s\Poset_n$ and $y \in s\Poset$ with $x \leq y$ or $y \leq x$, then $y \in s\Poset_n$. Thus two elements are comparable only if they have the same number of heads. 

We will now define a relation $\preceq$ on $s\Poset$ as follows. If $x = [\sigma_x] \in s\Poset$ and $y \in s\Poset$ such that $y = [\sigma_x \ast (F, id, 1_m)]$ for some elementary forest $F$, we write $x \preceq y$. That is, $x \preceq y$ if $y$ is obtained from $x$ via an elementary splitting. $\preceq$ is well defined with respect to dangling. If $x \preceq y$ and $x \neq y$, then we write $x \prec y$. Note that $\prec$ and $\preceq$ are not transitive, however, if $x \preceq z$ and $x \leq y \leq z$ then $x \preceq y$ and $y \preceq z$. 

We will now use the generalized class of diagrams just constructed to discuss the space which will be essential to our proof.

\section{The Stein complex}
\label{sec:def_stein_space}

We denote the simplicial complex with a $k$ simplex for every chain $x_0 < \cdots < x_k$ in $s\Poset_1$, called the \emph{geometric realization} of $s\Poset_1|$, by $|s\Poset_1|$ . A simplex is \emph{elementary} if $x_0 \prec x_k$.

\begin{definition}
	The \textit{Stein complex} $sX_{br}$ is the subcomplex of $|s\Poset_1|$ consisting of all elementary simplices. 
\end{definition}

We will prove $sX_{br}$ is contractible. We will first need to discuss what it means to be an interval in this context.

\begin{definition}
	Let $x, y \in |s\Poset_1|$. We define the closed interval $s[x,y]_{br} \defeq \{z | x \leq z \leq y\}$ and obtain the half open intervals $s(x,y]_{br}$ and $s[x,y)_{br}$ along with the open interval $s(x,y)_{br}$ by restricting the appropriate inequalities. If $x \preceq y$ then we call the interval \textit{elementary}, so $|s[x,y]_{br}|$ is contained in $sX_{br}$. 
\end{definition}

We will denote the unbraided version of $s(x,y)_{br}$ by $s(x,y)$, i.e. we replace the braid with a permutation. This interval is in its own poset, however, that is beyond the intent of this paper. The only thing we will need to know about $s(x,y)$ is that its geometric realization $|s(x,y)|$ is contractible by the following lemma.

\begin{lemma}\cite[Lemma 2.4]{Fluch13}
	\label{lem:sV_interval_contrble}
	For $x < y$ with $x \nprec y$, $|s(x,y)|$ is contractible. 
\end{lemma}

We wish to use Lemma \ref{lem:sV_interval_contrble} to prove that $|s(x,y)_{br}|$ is also contractible.

\begin{lemma}
	\label{lem:multicolored_int_contra}
	For $x \leq y$ where $x,y \in |s\Poset_1|$ and $x \npreceq y$, $|s(x,y)_{br}|$ is contractible.
\end{lemma}

\begin{proof}
	We will show that the poset $s(x,y)_{br}$ is isomorphic to $s(x,y)$, hence they have the same geometric realizations providing the desired result. 
	
	Recall that the equivalence relation for $s(x,y)_{br}$ is braiding the feet or ``mod dangling the feet'' and the equivalence relation for $s(x,y)$ is permuting the feet in a similar manner. For this proof we add the subscript $br$ to an equivalence class in $s(x,y)_{br}$ to distinguish it from an equivalence class in $s(x,y)$. Consider the function $\phi: s(x,y)_{br} \to s(x,y)$ which sends $z = [x \ast (F, \id, 1_n)]_{br}$ where $F$ is an elementary forest to $\phi(z) \defeq [x \ast (F, \id, 1_n)]$. Note while we write $x$ as part of $\phi(z)$ in the previous sentence we mean the permutation version of $x$ as $\phi(z)$ is in $s(x,y)$. Due to the equivalence relation on $s\Poset_1$ we see that if two elements of $s(x,y)_{br}$ are the same up to dangling then their images will be the same up to permutation, hence the function is well defined. Moreover, we get surjectivity by construction as every element of $s(x,y)_{br}$ was constructed by braiding an element of $s(x,y)$. 
	
	We must prove the map $\phi$ is injective to finish the proof. Consider two elements, say $z = [x \ast (F, \id, 1_n)]_{br}$ and $z' = [x \ast (F', \id, 1_{n'})]_{br}$.  Assuming $\phi(z) = \phi(z')$ we have $[x \ast (F, \id, 1_n)] = [x \ast (F', \id, 1_{n'})]$. By left cancellation, we have $[(F, \id, 1_n)] = [(F', \id, 1_{n'})]$. Therefore, $[(F, \sigma, F')]$ for some $\sigma \in S_n$ is trivial, i.e. it is the identity equivalence class. Since $[(F, \sigma, F')]$ is the identity, we know that there exists a sequence of expansions, reductions, and applications of the cross relation that turn $(F,\id,F)$ into $(F,\sigma,F')$. We can do the corresponding relations in the braided version to $[(F,\id,F)]$ to obtain $[(F,\beta,F')]_{br}$ for some braid $\beta$. That is, there exists a braid $\beta$ such that $[(F, \beta, F')]_{br}$ is the identity in the braided world. Therefore $[(F,\id, 1_n)]_{br} = [(F', \id, 1_n)]_{br}$, hence we have injectivity of $\phi$.	
	
	Therefore, we have $s(x,y)_{br}$ is isomorphic to $s(x,y)$. Hence $|s(x,y)_{br}|$ is contractible.
\end{proof}
\begin{corollary}
	\label{cor:sX_br_contractible}
	$sX_{br}$ is contractible.
\end{corollary}
\begin{proof}
	Observe that $|s\Poset_1|$ is contractible since directed posets are contractible. We can build from $sX_{br}$ to $|s\Poset_1|$ by attaching new subcomplexes, as has been done before for the Stein complexes in \cite{Bux16, Fluch13}, and we claim that this never changes the homotopy type, so $sX_{br}$ is contractible. 
	
	Given a closed interval $s[x,y]_{br}$, define $r(s[x,y]_{br}) := f(y) - f(x)$. We will attach closed intervals which are directed subposets, hence contractible subcomplexes, to $|s[x,y]_{br}|$ for $x \npreceq y$ to $sX_{br}$ in increasing order of $r$-value so that we attach it along $|s[x,y)_{br}|\cup |s(x,y]_{br}|$. That is, we identify the part of $|s[x,y]_{br}|$ which is already in $sX_{br}$ with itself. But this is the suspension of $|s(x,y)_{br}|$ because we have $x$ and $y$ as lower and upper bounds of $s(x,y)_{br}$ respectively, so attaching the intervals in this way makes two cone points. By Lemma \ref{lem:multicolored_int_contra}, $|s(x,y)_{br}|$ is contractible and the suspension of a contractible space is contractible.
	
	We conclude that attaching $|s[x,y]_{br}|$ to $|s\Poset_1|$ does not change the homotopy type, thus $|s\Poset_1|$ being contractible implies $sX_{br}$ is contractible.
\end{proof}

We will now define a filtration of the Stein complex. Recall the height of a vertex, denoted by $f:s\spraige \to \N$, is the number of feet it has, so its height increases when you add a split. 

\begin{definition}
For $n \geq 1$, the sublevel set $sX_{br}^{\leq n}$ is the full subcomplex of the Stein complex $sX_{br}$ spanned by all vertices of height at most $n$.	Moreover, $sX_{br}^{<n}$ is the full subcomplex of $sX_{br}$ spanned by the vertices $x$ such that its height is less than $n$.
\end{definition} 

This defines a natural filtration $(sX_{br}^{\leq n})_{n \in \N}$ of $sX_{br}$. We will show that the orbit space of the action of $s\Vbr$ on $sX_{br}^{\leq n}$ is compact, i.e. there are only finitely many orbits of simplices in the sublevel set $sX_{br}^{\leq n}$. 

\begin{lemma}[Cocompactness]
	\label{lem:sX_{br}_cocompact}
	For each $n \ge 1$ the sublevel set $sX_{br}^{\leq n}$ is finite modulo $s\Vbr$.
\end{lemma}

\begin{proof}
	Note that for $k \geq 1$, $s\Vbr$ acts on the left transitively on the set of $s$-colored $(1,k)$-spraiges. Thus there exists for each $1 \leq k \leq n$ one orbit of vertices $x$ in $sX_{br}^{\leq n}$ with $f(x) = k$. Given a vertex $x$ with $f(x) = k$, since there are only finitely many ways to get from $x$ to a vertex with $n$ feet,  there exists only finitely many simplices $\sigma_1, \dots, \sigma_r$ in the sublevel set $sX_{br}^{\leq n}$ that have $x$ as the vertex of the simplex with the minimum $f$ value of the simplex. If $\sigma$ is a simplex in $sX_{br}^{\leq n}$ such that its vertex with the minimum height is in the same orbit as $x$, then $\sigma$ must be in the same orbit as $\sigma_i$ for some $1 \leq i \leq r$. That is, when the group acts on $\sigma$ it is also acting on $x$ which in turn shows $\sigma$ is in the orbit of one of the previously mentioned simplices. It follows that there can only be finitely many orbits of simplices in the sublevel set $sX_{br}^{\leq n}$.
\end{proof}

\begin{lemma}[Vertex Stabilizers]
	\label{lem:sX_br_vertex_stab}
	Let $x$ be a vertex in $sX_{br}$ with $f(x) = n$. The stabilizer $\Stab_{s\Vbr}(x)$ is isomorphic to $B_n$.
\end{lemma}
\begin{proof}
	Firstly, identify $B_n$ with its image under the inclusion $B_n \to s\spraige_{n,n}$ via sending $b$ to $(1_n, b, 1_n)$. Let $g \in \Stab_{s\Vbr}(x)$ and fix $\sigma \in s\spraige_{1,n}$ with $x = [\sigma]$. By the definition of the stabilizer of a group we have $gx = x$, i.e. $[g \ast \sigma] = [\sigma]$. Thus $[\sigma^{-1} \ast g \ast \sigma] = [1_n]$, and so due to coset equality, $\sigma^{-1} \ast g \ast \sigma \in B_n$. Now define a map \\
	
	\centering
	\begin{tabular}{ccc}
		$\psi: \Stab_{s\Vbr}(x)$ & $\to$ & $B_n$\\
		$g$ & $\to$ & $\sigma^{-1} \ast g \ast \sigma$\\
	\end{tabular}	
	
	This map has an inverse $b \to \sigma \ast b \ast \sigma^{-1}$. Thus $\psi$ is an isomorphism. 
\end{proof}

\begin{lemma}[Simplex Stabilizers]
	\label{lem:sX_br_cell_stab}
	Let $\tau = \{x_0 < x_1 < \cdots < x_k\}$ be a $k$-simplex in $sX_{br}$, $x_0 = [\sigma]$, and $f(x_0) = n$. Then the stabilizer of $\tau$ is of type $\F_\infty$.
\end{lemma}

\begin{proof}
	For all $g \in s\Vbr$, $g \in \Stab_{s\Vbr}(\tau)$ if and only if $g$ stabilizes each $x_i$ for $0 \leq i \leq k$. Thus $\Stab_{s\Vbr}(\tau)$ is a subgroup of $\Stab_{s\Vbr}(x_0) \cong B_n$. 
	
	We now wish to prove that $PB_n$, the pure braid group on $n$ strands, stabilizes the simplex $\tau$. Let $y = [\sigma \ast \lambda]$ where $\lambda = (F,\id, 1_p)$ and $F$ is an elementary forest with $n$ roots and $p$ feet. Define $b_g \defeq \sigma^{-1} \ast g \ast \sigma \in B_n$ and specialize $b_g$ to be a pure braid, so $b_g \in PB_n$. We know that $PB_n = ker(B_n \to S_n)$ so pure braids are mapped to the identity permutation in $S_n$ being each strand always ends back where it started, hence $b_g \in B_n^J$ for all $J \subseteq \{1, 2, \dots, n\}$ where $J$ is the subset of indices that are do not move, where $B_n^J$ is as defined in \cite{Bux16}. Thus $[b_g \ast \lambda] = [\lambda]$ and by left multiplication of $\sigma$ on both sides we have $[\sigma \ast b_g \ast \lambda] = [\sigma \ast \lambda]$. By the definition of $b_g$ we get $[\sigma \ast (\sigma^{-1} \ast g \ast \sigma) \ast \lambda] = [\sigma \ast \lambda]$. After simplification we get $[g \ast \sigma \ast \lambda] = [\sigma \ast \lambda]$, hence $gy = y$. Therefore $g$ stabilizes $y$. As $x < y$ was arbitrary, we have shown that $g$ stabilizes $x_i$ for all $1 \leq i \leq k$. Hence $g$ stabilizes $\tau$, i.e.  $\psi^{-1}(b_g) \in \Stab_{s\Vbr}(\tau)$. Therefore, the simplex stabilizer contains $PB_n$, i.e. $\psi^{-1}(PB_n)  < \Stab_{s\Vbr}(\tau) < \Stab_{s\Vbr}(x_0) \cong \psi^{-1}(B_n)$. Since $S_n$ is finite, then $PB_n$ has finite index in $B_n$. Hence $\Stab_{s\Vbr}(\tau)$ has finite index in $\Stab_{s\Vbr}(x_0)$. Moreover, finite index subgroups of groups of type $\F_\infty$ are of type $\F_\infty$, hence we can conclude the simplex stabilizer is of type $\F_\infty$.

	%
	%

\end{proof}


The filtration $(sX_{br}^{\leq n})_{n \in \N}$ of $sX_{br}$ has satisfied all but one of the hypotheses of what is known as  Brown's Criterion which we will now state, namely that
the connectivity of the pair $(sX_{br}^{\le n+1},sX_{br}^{\le n})$ tends to $\infty$ as $n$ tends to $\infty$.

\begin{brownscriterion}\cite[Corollary 3.3]{brown87}
	\label{browns_criterion}
	Let $G$ be a group and $X$ a contractible $G$-CW-complex such that the stabilizer of every cell is of type $\F_\infty$.  Let $\{X_j\}_{j\ge 1}$ be a filtration of $X$ such that each $X_j$ is finite $\textnormal{mod}$ $G$. Suppose that the connectivity of the pair $(X_{j+1},X_j)$ tends to $\infty$ as $j$ tends to $\infty$.  Then $G$ is of type $\F_\infty$.
\end{brownscriterion}

The property of being type $\F_\infty$ and the connectivity properties of CW-pairs $(X_{j+1}, X_j)$ where $X_j \subseteq X_{j+1}$ are closely related. A standard tool for working with such a relationship is Bestvina-Brady Morse Theory, introduced by Bestvina and Brady in \cite{bestvina_brady_97}. The main ideas and results of \cite{bestvina_brady_97} will be the discussion of the rest of this section. This is also the key tool we will need in order finish proving our main result.

Let $X$ be an affine cell complex as in \cite{Bux16}. For our purposes, we merely need $X$ to be a simplicial complex. 
We need a Morse function on $X$ in the sense of \cite{bestvina_brady_97}. Recall a \emph{Morse function} is a map $h:X \rightarrow \R$ that restricts to a non-constant affine map on each positive dimensional cell of $X$; that is $h$ restricted to an $n$-cell $C$ of $X^{(n)}$ for $n \geq 1$ is a nonconstant affine map, and $h(X^{(0)})$ is a discrete set in $\R$. Note that each cell has a unique vertex maximizing $h$. Call $h(x)$ the \textit{height} of $x$ for vertices $x$. We can now define the essential notions of the link of a vertex and the corresponding descending link of a vertex.

\begin{definition}
	Given that $y$ is an element of $X^{(0)}$, the \textit{link} of $y$ is the space of directions out of $y$ into $X$ and is denoted by $lk(y)$. Moreover, the \textit{descending link of $y$} is the space of directions along which the given Morse function $h$ goes down and is denoted by $lk_{\downarrow}(y)$.
\end{definition}

In particular, if $X$ is the geometric realization of a poset and $h$ is a Morse function on $X$, then $\lk(y)$ is homeomorphic to the realization of the subposet of all $x$ in $X$ such that $x < y$ or $y < x$. Moreover, $\dlk(y)$ is homeomorphic to the realization of the poset of all $x$ such that $x \leq y$ or $y \leq x$ and $h(x) \le h(y)$. More details on descending links and links in general can be found in \cite{bestvina_brady_97}, and the following Morse Lemma is a consequence of \cite[Corollary 2.6]{bestvina_brady_97}.

\begin{morselemma}
\label{lem:morse_lemma}
Let $X$ be a simplicial complex and $h: X \to \R$ be a Morse function. 
 \begin{enumerate}
	\item Suppose that for any vertex $y$ with $h(y)=t$, $\dlk(y)$
	is $(k-1)$-connected.  Then the pair $(Y^{\le t},Y^{<t})$ is
	$k$-connected, that is, the inclusion $Y^{<t}\hookrightarrow
	Y^{\le t}$ induces an isomorphism in~$\pi_j$ for $j<k$, and an epimorphism in $\pi_k$.

	\item Suppose that for any vertex $y$ with $h(y)\ge t$,
	$\dlk(y)$ is $(k-1)$-connected.  Then $(Y^{\leq t},Y^{< t})$ is
	$k$-connected.
 \end{enumerate}
\end{morselemma}

\medskip

Every simplex of $sX_{br}$ has a unique vertex maximizing the number of feet function $f$, so $f$ extends to a Morse function on $|sX_{br}|$. Hence we can inspect the connectivity of the pair $(sX_{br}^{\le n},sX_{br}^{< n})$ by looking at descending links with respect to $f$. In Section \ref{sec:desc_link_conn}, we will describe a convenient model for the descending links which will include discussing how if $f(x) = n$ then $\lk_{\downarrow}(x) \cong s\elbraigecpx_n$, and then analyze the high connectivity of the descending links. Furthermore, later in Section \ref{sec:desc_link_conn} we will prove that $s\elbraigecpx_n$ is highly connected in
Corollary \ref{cor:multicolored_desc_link_conn}. Our proof relies on a complex that we call the \emph{multicolored matching complex on a surface}, which we will define and analyze in the next section. 


\section{Multicolored matching complexes on surfaces}
\label{sec:surfaces}

First we will mention what a matching complex on a graph and a multicolored matching arc complex on a surface are as they will be important later. We will finish the section by proving the aforementioned high connectivity result about multicolored matching arc complexes on a surface.

\begin{definition}[Matching complex of a graph]
	Let $\Gamma$ be a graph.  The \textit{matching complex} $\match(\Gamma)$ of $\Gamma$ is the simplicial complex with a $k$-simplex for every collection~$\{e_0,\dots,e_k\}$ of $k+1$ pairwise disjoint edges, with the face relation given by passing to subcollections.
\end{definition}

We will now discuss colored arcs and multicolored matching arc complexes. Informally we consider arcs defined as paths between distinct endpoints and assign colors to the arcs. As in previous sections, we will $s \geq 1$ to be the number of desired number of colors. Let $S$ denote a connected surface, with boundary $\partial S$, and $P$ a finite set of points in $S\backslash \partial S$. The most important case will be $S$ as a disk.

\begin{definition}
	By a \textit{colored arc}, we mean a simple path in $S \backslash \partial S$ that intersects $P$ precisely at its endpoints, whose endpoints are distinct, and which is assigned a color between 1 and $s$.
\end{definition} 

A collection of arcs that are disjoint except possibly at their endpoints is called an arc system. For an illustration of a multicolored arc system, see Figure \ref{fig:multicolored_arc_complex_link}. Formally we have the following:

\begin{definition}[Multicolored arc systems]
	Let $\{\alpha_0, \dots, \alpha_k\}$ be a collection of colored arcs. If the $\alpha_i$ are all disjoint from each other except possibly at their endpoints, and if no $\alpha_i$ and $\alpha_j$ are homotopic relative $P$, we call $\{\alpha_0, \dots, \alpha_k\}$ a \textit{multicolored arc system}. The homotopy classes, relative $P$, of multicolored arc systems form the simplices of a simplicial complex, with the face relation given by passing to subcollections of arcs. 
\end{definition}

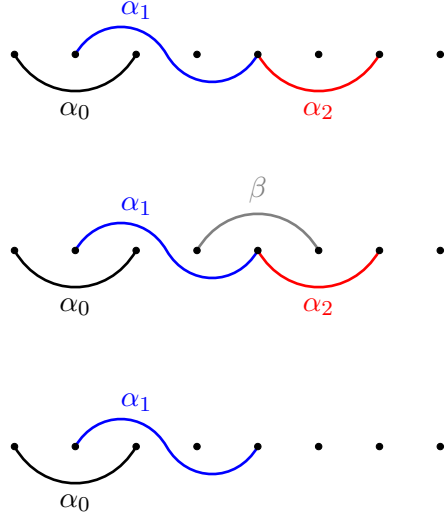
\begin{figure}[t]
     	\centering
     	\begin{tikzpicture}[scale=0.8]
     	\draw[line width=1pt]
     	(0,0) to [out=-60, in=240,looseness=1.2]  
     	node[pos=0.5, below] {$\alpha_{0}$} (2,0);
     	\draw[line width=1pt, blue]
     	(1,0) to [out= 60, in=120,looseness=1.2] 
     	node[pos=0.65, above] {$\alpha_{1}$} (2.5,0)
     	to [out=-60, in=240,looseness=1.2] (4,0);
     	\draw[line width=1pt, red]
     	(4,0) to [out=-60, in=240,looseness=1.2]  
     	node[pos=0.5, below] {$\alpha_{2}$} (6,0);
     	
     	\filldraw 
     	(0,0) circle (1.5pt)
     	(1,0) circle (1.5pt)
     	(2,0) circle (1.5pt)
     	(3,0) circle (1.5pt)
     	(4,0) circle (1.5pt)
     	(5,0) circle (1.5pt)
     	(6,0) circle (1.5pt)
     	(7,0) circle (1.5pt);  
     	
     	\begin{scope}[yshift=-3.25cm]
     	\draw[line width=1pt]
     	(0,0) to [out=-60, in=240,looseness=1.2]  
     	node[pos=0.5, below] {$\alpha_{0}$} (2,0);
     	\draw[line width=1pt, blue]
     	(1,0) to [out= 60, in=120,looseness=1.2] 
     	node[pos=0.65, above] {$\alpha_{1}$} (2.5,0)
     	to [out=-60, in=240,looseness=1.2] (4,0);
     	\draw[line width=1pt, red]
     	(4,0) to [out=-60, in=240,looseness=1.2]  
     	node[pos=0.5, below] {$\alpha_{2}$} (6,0);
     	\draw[line width=1pt,gray]
     	(3,0) to [out=60, in=120,looseness=1.2]  
     	node[pos=0.5, above] {$\beta$} (5,0);
     	
     	\filldraw 
     	(0,0) circle (1.5pt)
     	(1,0) circle (1.5pt)
     	(2,0) circle (1.5pt)
     	(3,0) circle (1.5pt)
     	(4,0) circle (1.5pt)
     	(5,0) circle (1.5pt)
     	(6,0) circle (1.5pt)
     	(7,0) circle (1.5pt);  
     	\end{scope} 
     	
     	\begin{scope}[yshift=-6.5cm]
     	\draw[line width=1pt]
     	(0,0) to [out=-60, in=240,looseness=1.2]  
     	node[pos=0.5, below] {$\alpha_{0}$} (2,0);
     	\draw[line width=1pt, blue]
     	(1,0) to [out= 60, in=120,looseness=1.2] 
     	node[pos=0.65, above] {$\alpha_{1}$} (2.5,0)
     	to [out=-60, in=240,looseness=1.2] (4,0);
     	
     	\filldraw
     	(0,0) circle (1.5pt)
     	(1,0) circle (1.5pt)
     	(2,0) circle (1.5pt)
     	(3,0) circle (1.5pt)
     	(4,0) circle (1.5pt)
     	(5,0) circle (1.5pt)
     	(6,0) circle (1.5pt)
     	(7,0) circle (1.5pt);  
     	\end{scope}
     	\end{tikzpicture}
     	\caption{From top to bottom: a multicolored arc system $\sigma$, a simplex after adding a colored arc to $\sigma$, and a simplex after taking away a colored arc from $\sigma$.}
     	\label{fig:multicolored_arc_complex_link}
\end{figure}

Recall that $K_n$ is the complete graph on $n$ vertices. 

\begin{definition}[Multicolored matching complex on a surface]
	Let $s\matcharc(K_n)$ be the complex whose simplices are given by multicolored arc systems whose arcs are pairwise disjoint including at their endpoints. We call $s\matcharc(K_n)$ the \textit{multicolored matching arc complex} on $(S,P)$ corresponding to $K_n$. 
\end{definition}

Observe that if $s = 1$, then $s\matcharc(K_n) = \matcharc(K_n)$ is the subcomplex from \cite{Bux16}. Define $\nu(n) \defeq  \floor{\frac{n-2}{3}}$ where $n \in \N$. It was proven in \cite{Bux16} that $\matcharc(K_n)$ is $(\nu(n)-1)$-connected  We will now prove the higher connectivity of $s\matcharc(K_n)$ for arbitrary $s$. 

\begin{theorem}
	\label{thrm:multicolored_surface_matching_conn}
	The complex $s\matcharc(K_n)$ is $(\nu(n)-1)$-connected.
\end{theorem}

\begin{proof}
	Let $\sigma$ be a simplex in $\matcharc(K_n)$ where $\dim(\sigma) = k$. Define a map $\phi: s\matcharc(K_n) \to \matcharc(K_n)$ which sends a multicolored arc to the unique arc with the same endpoints. In other words, $\phi(\sigma)$ for $\sigma \in s\matcharc(K_n)$ is $\sigma$ after forgetting all of its colors. Observe that while $s\matcharc(K_n)$ and $\matcharc(K_n)$ are not posets, we can barycentrically divide which makes the barycentric subdivisions of $s\matcharc(K_n)$ and $s\matcharc(K_n)$ geometric realizations of posets. This is sufficient enough to work for \cite{quillen78}[Theorem 9.1]. We know by \cite[Theorem 3.8]{Bux16} that $\matcharc(K_n)$ is $(\nu(n)-1)$-connected. It suffices to show that the link $\lk(\sigma)$ is $(\nu(n) - k - 2)$-connected and the fiber $\phi^{-1}(\sigma)$ is $(k - 1)$-connected. One can tell that the links of $\matcharc(K_n)$ are matching arc complexes themselves, that is, $\lk(\sigma)$ is isomorphic to $\matcharc(K_{n-2(k+1)})$. It follows from \cite{Bux16}[Theorem 3.8] that $\lk(\sigma)$ is $(\nu(n-2k-2) - 1)$-connected. Observe the following inequality:
	
		\begin{tabular}{lll}
	$\nu(n - 2k - 2) - 1$ & $=$ & $\floor{\frac{(n -2k - 2)-2}{3}} - 1$ \\
						 & $=$ & $\floor{\frac{(n-2) -2k - 2}{3}} - 1$ \\
						 & $\geq$ & $\floor{\frac{(n-2) -2k}{3}} - 2$ \\
						 & $\geq$ & $\floor{\frac{n-2}{3}} - k - 2$ \\
						 & $=$  & $\nu(n) - k - 2$ \\
		\end{tabular} 
	
	So the link is $(\nu(n) - k - 2)$-connected as desired.  Observe that the fiber $\phi^{-1}(\sigma)$ is the join of the fibers of the $k+1$ vertices of $\sigma$. Therefore $\phi^{-1}(\sigma)$ is $(k - 1)$-connected being the join of $k+1$ nonempty things. We can now apply \cite[Theorem 9.1]{quillen78} and conclude that $s\matcharc(K_n)$ is $(\nu(n)-1)$-connected.
\end{proof}


\section{Descending links in the Stein complex}
\label{sec:desc_link_conn}

We now return to the Stein complex $sX_{br}$ and inspect the descending links of vertices with respect to the height function $f$. 
There is a relationship between the descending link $\dlk(x)$ of a vertex $x$ where $f(x) = n$ in the Stein complex $sX_{br}$ and $s\elbraigecpx_n$. In particular, thanks to left cancellation $\dlk(x)$ is isomorphic to the simplicial complex $s\elbraigecpx_n$. See Figure \ref{fig:desc_lk_to_sEBn} for a visualization of this correspondence. It is standard to draw multicolored braiges as emerging from a horizontal line as there are no splits and as a visual reminder of this correspondence. 

It is also important to define $sL_n$ to be the linear graph with $n+1$ vertices labeled $v_0, \dots, v_n$ and exactly $s$ edges $e_{i, j}$ between each vertex with $Ends(e_{i,j}) = \{v_{i-1}, v_i\}$ for $1 \leq i \leq n$, $1 \leq j \leq s$. 
Let $\match(sL_{n-1})$ be the matching complex on $sL_{n-1}$. 

\begin{figure}
	\centering
	\begin{tikzpicture}
		
		
		\coordinate (a') at (0,0);
		\coordinate (e') at (-2,-2);
		\coordinate (f') at (0,-4);
		\coordinate (g') at (2,-2);
		\coordinate (h') at (-2,0);
		\coordinate (m') at (-1,-2.5);
		
		
		\filldraw[lightgray]
		(e') -- (m') -- (f') -- (e')
		(a') -- (e') -- (m') -- (f') -- (g') -- (a');
		
		\draw
		(a') -- (e')
		(a') -- (g')
		(e') -- (f')
		(f') -- (g')
		(f') -- (m')
		(m') -- (a')
		(m') -- (e')
		(m') -- (g');
		\draw[gray]
		(a') -- (f');
		
		\filldraw (a') circle (1pt);
		\filldraw (e') circle (1pt);
		\filldraw (f') circle (1pt);
		\filldraw (g') circle (1pt);
		\filldraw (m') circle (1pt);
		%
		%
		\begin{scope}[xshift=-1cm,yshift=-1cm]
			\coordinate (z) at (1, 1.9);
			\coordinate (z') at (1, 1.7);
			\coordinate (a) at (.75,1.4);
			\coordinate (b) at (1.25,1.4);
			\coordinate (x) at (1,1.52);
			\coordinate (c) at (0.85, 1.4);
			\coordinate (d) at (0.85, 1.2);
			\coordinate (e) at (0.95, 1.4);
			\coordinate (f) at (0.95, 1.2);
			\coordinate (g) at (1.05, 1.4);
			\coordinate (h) at (1.05, 1.2);
			\coordinate (i) at (1.15, 1.4);
			\coordinate (j) at (1.15, 1.2);
			\draw
			(a) to [out=90, in=90, looseness=2] (b) -- (a)
			(z) -- (z')
			(c) -- (d)   
			(e) -- (f)   
			(g) -- (h)   
			(i) -- (j);
			\node at (x) {$x$};
		\end{scope}
		
		
		\begin{scope}[xshift=-3.25cm, yshift=-3cm]
			\coordinate (z) at (1, 1.9);
			\coordinate (z') at (1, 1.7);
			\coordinate (a) at (.75,1.4);
			\coordinate (b) at (1.25,1.4);
			\coordinate (x) at (1,1.52);
			\coordinate (c) at (0.85, 1.4);
			\coordinate (d) at (0.85, 1.2);
			\coordinate (e) at (0.95, 1.4);
			\coordinate (f) at (0.95, 1.2);
			\coordinate (g) at (1.05, 1.4);
			\coordinate (h) at (1.05, 1.2);
			\coordinate (i) at (1.15, 1.4);
			\coordinate (j) at (1.15, 1.2);
			\coordinate (k) at (1.1, 1.1);
			\draw
			(a) to [out=90, in=90, looseness=2] (b) -- (a)
			(z) -- (z')
			(c) -- (d)   
			(e) -- (f)   
			(g) -- (h)   
			(i) -- (j);
			\node at (x) {$x$};
			\draw[blue]
			(h) -- (k) -- (j);
		\end{scope}
		
		\begin{scope}[xshift=1.25cm, yshift=-3cm]
			\coordinate (z) at (1, 1.9);
			\coordinate (z') at (1, 1.7);
			\coordinate (a) at (.75,1.4);
			\coordinate (b) at (1.25,1.4);
			\coordinate (x) at (1,1.52);
			\coordinate (c) at (0.85, 1.4);
			\coordinate (d) at (0.85, 1.2);
			\coordinate (e) at (0.95, 1.4);
			\coordinate (f) at (0.95, 1.2);
			\coordinate (g) at (1.05, 1.4);
			\coordinate (h) at (1.05, 1.2);
			\coordinate (i) at (1.15, 1.4);
			\coordinate (j) at (1.15, 1.2);
			\coordinate (k) at (0.9, 1.1);
			\draw
			(a) to [out=90, in=90, looseness=2] (b) -- (a)
			(z) -- (z')
			(c) -- (d)   
			(e) -- (f)   
			(g) -- (h)   
			(i) -- (j);
			\draw[blue]
			(d) -- (k) -- (f);
			\node at (x) {$x$};
		\end{scope}
		
		\begin{scope}[xshift=-2.5cm,yshift=-4.5cm]
			\coordinate (z) at (1, 1.9);
			\coordinate (z') at (1, 1.7);
			\coordinate (a) at (.75,1.4);
			\coordinate (b) at (1.25,1.4);
			\coordinate (x) at (1,1.52);
			\coordinate (c) at (0.85, 1.4);
			\coordinate (d) at (0.85, 1.2);
			\coordinate (e) at (0.95, 1.4);
			\coordinate (f) at (0.95, 1.2);
			\coordinate (g) at (1.05, 1.4);
			\coordinate (h) at (1.05, 1.2);
			\coordinate (i) at (1.15, 1.4);
			\coordinate (j) at (1.15, 1.2);
			\coordinate (k) at (1.1, 1.1);
			\coordinate (l) at (0.9, 1.1);
			\draw
			(a) to [out=90, in=90, looseness=2] (b) -- (a)
			(z) -- (z')
			(c) -- (d)   
			(e) -- (f)   
			(g) -- (h)   
			(i) -- (j);
			\node at (x) {$x$};
			\draw[blue]
			(d) -- (l) -- (f)
			(h) -- (k) -- (j);
		\end{scope}

		\begin{scope}[xshift=-0.6cm, yshift=-5.6cm]
			\coordinate (z) at (1, 1.9);
			\coordinate (z') at (1, 1.7);
			\coordinate (a) at (.75,1.4);
			\coordinate (b) at (1.25,1.4);
			\coordinate (x) at (1,1.52);
			\coordinate (c) at (0.85, 1.4);
			\coordinate (d) at (0.85, 1.2);
			\coordinate (e) at (0.95, 1.4);
			\coordinate (f) at (0.95, 1.2);
			\coordinate (g) at (1.05, 1.4);
			\coordinate (h) at (1.05, 1.2);
			\coordinate (i) at (1.15, 1.4);
			\coordinate (j) at (1.15, 1.2);
			\coordinate (k) at (1.1, 1.1);
			\coordinate (l) at (0.9, 1.1);
			\coordinate (m) at (1, .9);
			\draw
			(a) to [out=90, in=90, looseness=2] (b) -- (a)
			(z) -- (z')
			(c) -- (d)   
			(e) -- (f)   
			(g) -- (h)   
			(i) -- (j);
			\node at (x) {$x$};
			\draw[blue]
			(d) -- (l) -- (f)
			(h) -- (k) -- (j);
			\draw[red]
			(l) -- (m) -- (k);
		\end{scope}
		
		\node at (3.5,-2) {$\longleftrightarrow$};
		
		\coordinate (A) at (4.5, 0);
		\coordinate (B) at (8.5, 0);
		\coordinate (C) at (6.5, -1.5);
		\coordinate (D) at (6.5, -4);
		
		\fill[lightgray]
		(A) -- (C) -- (D) -- (A)
		(B) -- (C) -- (D) -- (B);
		\filldraw
		(A) circle (1.5pt)
		(B) circle (1.5pt)
		(C) circle (1.5pt)
		(D) circle (1.5pt);
		\draw
		(A) -- (C)
		(A) -- (D)
		(B) -- (C)
		(B) -- (D)
		(C) -- (D);
		%
		
		\begin{scope}[xshift=2.5cm, yshift=-1.75cm, scale=2]
			\coordinate (a) at (.75,1.4);
			\coordinate (b) at (1.25,1.4);
			\coordinate (x) at (1,1.52);
			\coordinate (c) at (0.85, 1.4);
			\coordinate (d) at (0.85, 1.2);
			\coordinate (e) at (0.95, 1.4);
			\coordinate (f) at (0.95, 1.2);
			\coordinate (g) at (1.05, 1.4);
			\coordinate (h) at (1.05, 1.2);
			\coordinate (i) at (1.15, 1.4);
			\coordinate (j) at (1.15, 1.2);
			\coordinate (k) at (1.1, 1.1);
			\coordinate (l) at (0.9, 1.1);
			\coordinate (m) at (1, .9);
			\draw
			(b) -- (a)
			(c) -- (d)   
			(e) -- (f)   
			(g) -- (h)   
			(i) -- (j);
			\draw[blue]
			(d) -- (l) -- (f);
		\end{scope}
		
		\begin{scope}[xshift=6.9cm, yshift=-1.75cm, scale=2]
			\coordinate (a) at (.75,1.4);
			\coordinate (b) at (1.25,1.4);
			\coordinate (x) at (1,1.52);
			\coordinate (c) at (0.85, 1.4);
			\coordinate (d) at (0.85, 1.2);
			\coordinate (e) at (0.95, 1.4);
			\coordinate (f) at (0.95, 1.2);
			\coordinate (g) at (1.05, 1.4);
			\coordinate (h) at (1.05, 1.2);
			\coordinate (i) at (1.15, 1.4);
			\coordinate (j) at (1.15, 1.2);
			\coordinate (k) at (1.1, 1.1);
			\coordinate (l) at (0.9, 1.1);
			\coordinate (m) at (1, .9);
			\draw
			(b) -- (a)
			(c) -- (d)   
			(e) -- (f)   
			(g) -- (h)   
			(i) -- (j);
			\draw[blue]
			(h) -- (k) -- (j);
		\end{scope}
		
		\begin{scope}[xshift=4.5cm, yshift=-3.1cm, scale=2]
			\coordinate (a) at (.75,1.4);
			\coordinate (b) at (1.25,1.4);
			\coordinate (x) at (1,1.52);
			\coordinate (c) at (0.85, 1.4);
			\coordinate (d) at (0.85, 1.2);
			\coordinate (e) at (0.95, 1.4);
			\coordinate (f) at (0.95, 1.2);
			\coordinate (g) at (1.05, 1.4);
			\coordinate (h) at (1.05, 1.2);
			\coordinate (i) at (1.15, 1.4);
			\coordinate (j) at (1.15, 1.2);
			\coordinate (k) at (1.1, 1.1);
			\coordinate (l) at (0.9, 1.1);
			\coordinate (m) at (1, .9);
			\draw
			(b) -- (a)
			(c) -- (d)   
			(e) -- (f)   
			(g) -- (h)   
			(i) -- (j);
			\draw[blue]
			(d) -- (l) -- (f)
			(h) -- (k) -- (j);
		\end{scope}

		\begin{scope}[xshift=4.6cm, yshift=-7cm, scale=2]
			\coordinate (a) at (.75,1.4);
			\coordinate (b) at (1.25,1.4);
			\coordinate (x) at (1,1.52);
			\coordinate (c) at (0.85, 1.4);
			\coordinate (d) at (0.85, 1.2);
			\coordinate (e) at (0.95, 1.4);
			\coordinate (f) at (0.95, 1.2);
			\coordinate (g) at (1.05, 1.4);
			\coordinate (h) at (1.05, 1.2);
			\coordinate (i) at (1.15, 1.4);
			\coordinate (j) at (1.15, 1.2);
			\coordinate (k) at (1.1, 1.1);
			\coordinate (l) at (0.9, 1.1);
			\coordinate (m) at (1, .9);
			\draw
			(b) -- (a)
			(c) -- (d)   
			(e) -- (f)   
			(g) -- (h)   
			(i) -- (j);
			\draw[blue]
			(d) -- (l) -- (f)
			(h) -- (k) -- (j);
			\draw[red]
			(l) -- (m) -- (k);
		\end{scope}
	\end{tikzpicture}
	\caption{The correspondence between $\dlk(x)$ and $s\elbraigecpx_n$.}
	\label{fig:desc_lk_to_sEBn}
\end{figure}

%
%
%

\begin{observation}
	Very elementary forests with $n$ leaves correspond bijectively to simplices of $\match(sL_{n-1})$. This is given by identifying a colored caret to the same colored edge. See Figure \ref{fig:matchings_to_multicolored_forests} for an example of this correspondence.
\end{observation}

\subsection{Connectivity of $s\velbraigecpx_{n}$}

We will now construct a projection $s\velbraigecpx_n \to s\matcharc(K_n)$ to prove the connectivity of $s\velbraigecpx_n$ using the connectivity of $s\matcharc(K_n)$ and the tools of \cite{quillen78}. Similar to what was done in \cite{Bux16}, we will denote a very elementary multicolored $n$-braige by $(b, \Gamma)$, where $b \in B_n$ and $\Gamma$ is a simplex in $\match(sL_{n-1})$. As we have the equivalence under dangling, we denote the equivalence class by $[(b, \Gamma)]$.

\begin{figure}[t]
	\centering
		\begin{tikzpicture}
			
		\coordinate (a) at (0,0);
		\coordinate (b) at ($(a)+(0.5,-1)$);
		\coordinate (c) at ($(a)+(1,0)$);
		\coordinate (d) at ($(a)+(2,0)$);
		\coordinate (e) at ($(a)+(2.5,-1)$);
		\coordinate (f) at ($(a)+(3,0)$);
		\coordinate (g) at ($(a)+(4,0)$);
		\coordinate (h) at ($(a)+(5,0)$);
		\coordinate (i) at ($(a)+(5.5,-1)$);
		\coordinate (j) at ($(a)+(6,0)$);
		\coordinate (k) at ($(a)+(7,0)$);
		\coordinate (l) at ($(a)+(7.5,-1)$);
		\coordinate (m) at ($(a)+(8,0)$);
			
		\filldraw[red]
		(a) circle (1.5pt)
		(b) circle (1.5pt)
		(c) circle (1.5pt)
		(k) circle (1.5pt)
		(l) circle (1.5pt)
		(m) circle (1.5pt);
		\filldraw[blue]
		(d) circle (1.5pt)
		(e) circle (1.5pt)
		(f) circle (1.5pt)		
		(h) circle (1.5pt)
		(i) circle (1.5pt)
		(j) circle (1.5pt);
		\filldraw
		(g) circle (1.5pt);

		\draw[line width=1pt, red]
		(a) -- (b) -- (c)
		(k) -- (l) -- (m);
		
		\draw[line width=1pt, blue]
		(d) -- (e) -- (f)
		(h) -- (i) -- (j);
		
		\node[rotate=-90] at (4,-1.5) {$\mapsto$};
		
		\begin{scope}[yshift=-2.5cm]
		\coordinate (a) at (0,0);
		\coordinate (b) at ($(a)+(1,0)$);
		\coordinate (c) at ($(a)+(2,0)$);
		\coordinate (d) at ($(a)+(3,0)$);
		\coordinate (e) at ($(a)+(4,0)$);
		\coordinate (f) at ($(a)+(5,0)$);
		\coordinate (g) at ($(a)+(6,0)$);
		\coordinate (h) at ($(a)+(7,0)$);
		\coordinate (i) at ($(a)+(8,0)$);

		\filldraw[red]
		(a) circle (1.5pt)
		(b) circle (1.5pt)
		(h) circle (1.5pt)
		(i) circle (1.5pt);
		\filldraw[blue]
		(c) circle (1.5pt)
		(d) circle (1.5pt)
		(f) circle (1.5pt)
		(g) circle (1.5pt);
		\filldraw
		(e) circle (1.5pt);
		
		\draw[line width=1pt, red]
		(a) -- (b)   
		(h) -- (i);
		\draw[line width=1pt, blue]
		(c) -- (d)   
		(f) -- (g);
		\end{scope}
		\end{tikzpicture}
	\caption{An example of the bijective correspondence between very elementary multicolored forests with $9$ leaves and simplices of $\match(2L_8)$.}
	\label{fig:matchings_to_multicolored_forests}
\end{figure}
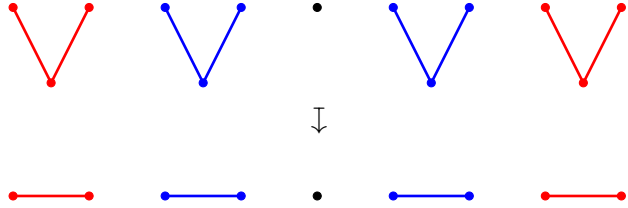

Now let $S$ be the unit disk and fix an embedding $sL_{n-1} \hookrightarrow S$ to be the map which sends a colored edge between two vertices $v_i$ and $v_{i+1}$ for $1 \leq i \leq n-1$ to the same coloring of a colored arc between the $i^{th}$ and $(i+1)^{st}$ puncture of $S$. Let $P$ be the image of the vertex set under this map, hence $P$ has $n$ points in $S$ labeled $v_1, \dots, v_n$. This set up allows us to consider the $s$-colored matching arc complex $s\matcharc(K_n)$ where we have an induced embedding of simplicial complexes $\match(sL_{n-1}) \hookrightarrow s\matcharc(K_n)$. We know from \cite{Birman74} that the braid group on $n$ strands $B_n$ is isomorphic to the mapping class group of the $n$-punctured disk $D_n$. Since $S\setminus P = D_n$, we have an action of $B_n$ on $s\matcharc(K_n)$. We will specialize to the right action as that follows from what we discussed for dangling multicolored $n$-braiges, and thus we introduce the notation $(\sigma)b$ to be the image of $\sigma$ under $b$ where $\sigma \in s\matcharc(K_n)$ and $b \in B_n$. 

We will now define our map $s\pi$ from $s\velbraigecpx_n$ to $s\matcharc(K_n)$. Viewing $\match(sL_{n-1})$ as a subcomplex of $s\matcharc(K_n)$ we can associate to a very elementary multicolored $n$-braige $(b, \Gamma)$ the multicolored arc complex $(\Gamma)b^{-1}$ in $s\matcharc(K_n)$. By construction with the above discussion, the map $(b, \Gamma) \to (\Gamma)b^{-1}$ is well defined on equivalence classes under dangling, so we obtain the following map:

\begin{definition}
	Using the above construction, define the following simplicial map: \[s\pi: s\velbraigecpx_n \to s\matcharc(K_n)\] \[[(b,\Gamma)] \to (\Gamma)b^{-1}.\] This map has a visualization of viewing merges as arcs, then ``combing straight'' the braid and seeing where the arcs are taken.
\end{definition}  

Note that $s\pi$ is surjective but not injective. This visualization can be seen in Figure \ref{fig:BMD-to-multi-arc-cplx}. Note that the resulting simplex $(\Gamma)b^{-1}$ of $s\matcharc(K_n)$ has the same dimension as the simplex $[(b,\Gamma)]$ of $s\velbraigecpx_n$. 
The following lemma and its corresponding proof are similar to the monocolored case where $s=1$. 

\begin{figure}[t]
	\centering
	\begin{tikzpicture}
		\clip(-.30,0.4) rectangle (1.80,-2.95);
		
		\draw[line width=3pt, white] (0.5,0) to [out=270, in=90] (0,-2);
		\draw[line width=1pt] (0.5,0) to [out=270, in=90] (0,-2);
		
		\draw[line width=3pt, white] (0,0) to [out=270, in=90] (1,-2);    
		\draw[line width=1pt] (0,0) to [out=270, in=90] (1,-2);    
		
		\draw[line width=3pt, white] (1,0) to [out=270, in=90] (0.5,-2);
		\draw[line width=1pt] (1,0) to [out=270, in=90] (0.5,-2);
		
		\draw[line width=1pt] (1.5,0) to (1.5,-2);
		
		\draw[line width=1pt] (0,-2) -- (0,-2.1);
		
		\draw[line width=1pt, red] (0,-2.1) -- (0.25,-2.5) -- (0.5,-2.1);
		
		\draw[line width=1pt] (0.5,-2.1) -- (0.5,-2.0);
		
		\draw[line width=1pt, blue] (1,-2.1) -- (1.25,-2.5) -- 
		(1.5,-2.1);
		
		\draw[line width=1pt] (1,-2) -- (1,-2.1)  
		(1.5,-2.1) -- (1.5,-2.0);
		
		\filldraw[red] 
		(0.25,-2.5) circle (1pt)
		(0.5,-2.1) circle (1pt)
		(0,-2.1) circle (1pt);
		\filldraw[blue]
		(1,-2.1) circle (1pt) 
		(1.25, -2.5) circle (1pt)
		(1.5,-2.1) circle (1pt);
		
		\draw[line width=1pt] (-0.25,0) -- (1.75,0);
		
	\end{tikzpicture}
	\qquad
	\begin{tikzpicture}
		\clip(-.30,0.4) rectangle (1.80,-2.95);
		\draw[line width=3pt, white] (0.5,0) to [out=270, in=90] (0,-2);
		\draw[line width=1pt] (0.5,0) to [out=270, in=90] (0,-2);
		
		\draw[line width=3pt, white] (0,0) to [out=270, in=90] (1,-2);    
		\draw[line width=1pt] (0,0) to [out=270, in=90] (1,-2);    
		
		\draw[line width=3pt, white] (1,0) to [out=270, in=90] (0.5,-2);
		\draw[line width=1pt] (1,0) to [out=270, in=90] (0.5,-2);
		
		\draw[line width=1pt] (1.5,0) to (1.5,-2);
		
		\draw[line width=1pt] (0,-2) -- (0,-2.5)   
		(0.5,-2.5) -- (0.5,-2.0);
		
		\draw[line width=1pt] (1,-2) -- (1,-2.5) 
		(1.5,-2.5) -- (1.5,-2.0);
		
		\draw[line width=1pt,red] 
		(0,-2.5) -- (0.5,-2.5);
		\draw[line width=1pt, blue]
		(1,-2.5) -- (1.5,-2.5); 
		\filldraw[red] 
		(0,-2.5) circle (1pt) -- (0.5,-2.5) circle (1pt);
		\filldraw[blue]
		(1,-2.5) circle (1pt) -- (1.5,-2.5) circle (1pt); 
		
		\draw[line width=1pt] (-0.25,0) -- (1.75,0);
		
	\end{tikzpicture}
	\qquad
	\begin{tikzpicture}
		\clip(-.30,0.4) rectangle (1.80,-2.95);
		\draw[line width=1pt,red] 
		(0,-2.5) to [out=-30, in=210] (1,-2.5);
		\draw[line width=1pt, blue]
		(0.5,-2.5) to [out=30, in=150] (1.5,-2.5); 
		\draw[line width=3pt,white] (1,-2) -- (1,-2.4);
		
		\draw[line width=3pt, white] (0.5,-1) to [out=270, in=90] (0,-2);
		\draw[line width=1pt] (0.5,0) -- (0.5,-1) to [out=270, in=90] (0,-2);
		
		\draw[line width=3pt, white] (0,-1) to [out=270, in=90] (0.5,-2);    
		\draw[line width=1pt] (0,0) -- (0,-1) to [out=270, in=90] (0.5,-2);    
		
		\draw[line width=3pt, white] (1,0) to [out=270, in=90] (1,-2);
		\draw[line width=1pt] (1,0) to [out=270, in=90] (1,-2);
		
		\draw[line width=3pt, white] (1.5,0) to (1.5,-2);
		\draw[line width=1pt] (1.5,0) to (1.5,-2);
		
		\draw[line width=1pt] (0,-2) -- (0,-2.5)   
		(0.5,-2.5) -- (0.5,-2.0);
		
		\draw[line width=1pt] (1,-2) -- (1,-2.5)
		(1.5,-2.5) -- (1.5,-2.0);
		\filldraw[red] 
		(0,-2.5) circle (1pt) (1,-2.5) circle (1pt);
		\filldraw[blue]
		(0.5,-2.5) circle (1pt) (1.5,-2.5) circle (1pt); 
		
		\draw[line width=1pt] (-0.25,0) -- (1.75,0);

	\end{tikzpicture}
	\qquad
	\begin{tikzpicture}
		\clip(-.30,0.4) rectangle (1.80,-2.95);
		\draw[line width=1pt,red]
		(-0.25,-2.5) to [out=-90, in=230] (1,-2.5)
		(-0.25,-2.5) to [out=90, in=100] (0.5,-2.5);
		
		\draw[line width=1pt, blue]
		(0.0,-2.5) to [out=300, in=225] (0.75,-2.5)
		(0.75,-2.5) to [out=45, in=150] (1.5,-2.5); 
		\draw[line width=3pt,white] (1,-2) -- (1,-2.4);
		\draw[line width=3pt,white] (0,-2) -- (0,-2.4);

		\draw[line width=3pt, white] (0.5,0) to [out=270, in=90] (0.5,-2);
		\draw[line width=1pt] (0.5,0) to [out=270, in=90] (0.5,-2);
		
		\draw[line width=3pt, white] (0,0) to [out=270, in=90] (0,-2);    
		\draw[line width=1pt] (0,0) to [out=270, in=90] (0,-2);    
		
		\draw[line width=3pt, white] (1,0) to [out=270, in=90] (1,-2);
		\draw[line width=1pt] (1,0) to [out=270, in=90] (1,-2);
		
		\draw[line width=3pt, white] (1.5,0) to (1.5,-2);
		\draw[line width=1pt] (1.5,0) to (1.5,-2);
		
		\draw[line width=1pt] (0,-2) -- (0,-2.5)   
		(0.5,-2.5) -- (0.5,-2.0);
		
		\draw[line width=1pt] (1,-2) -- (1,-2.5)
		(1.5,-2.5) -- (1.5,-2.0);
		\filldraw[blue] 
		(0,-2.5) circle (1pt)
		(1.5,-2.5) circle (1pt); 
		\filldraw[red]
		(1,-2.5) circle (1pt)
		(0.5,-2.5) circle (1pt);

		\draw[line width=1pt] (-0.25,0) -- (1.75,0);
		
	\end{tikzpicture}
	\caption{From multicolored braiges to multicolored arc systems. From left to right the pictures show the process of	``combing straight'' the braid.}
	\label{fig:BMD-to-multi-arc-cplx}
\end{figure}
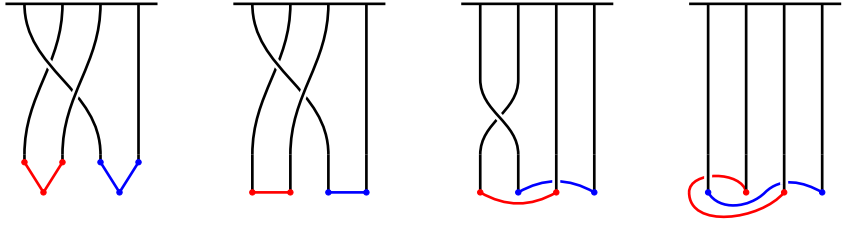

\begin{lemma}
	\label{lem:common_simplex_lemma}
	Let $\Gamma$ and $\Phi$ be simplices in $\match(sL_{n-1})$, such that $\Gamma$ has one edge and $\Phi$ has $e(\Phi)$ edges. Let $[(b, \Gamma)]$ and $[(c, \Phi)]$ be dangling elementary multicolored $n$-braiges. Suppose that their images under the map $\pi$ are contained in a simplex of $s\matcharc(K_n)$. Then there exists a simplex in $s\velbraigecpx_n$ that contains $[(b, \Gamma)]$ and $[(c, \Phi)]$.
\end{lemma}

\begin{proof}
	Assume that the simplex $[(b, \Gamma)]$ is not a face of $[(c, \Phi)]$. 
	
	There is a natural action of $B_n$ on $s\velbraigecpx_n$ by braiding the strands at the top, i.e. $b'[(c, \Phi')] = [(b'c', \Phi')]$. It follows that up to the action of $B_n$ without loss of generality we can assume $c = \id$, and $\Phi$ is a subgraph of $sL_{n-1}$ which has precisely $e(\Phi)$ edges connecting $j$ to $j+1$ for $j \in \{1, 3, \dots, 2e(\Phi) - 1\}$.
	
	Now there is a colored arc $\alpha$ representing $s\pi([(b, \Gamma)])$ that is disjoint from $\Phi$. Due to the disjointness, after dangling, we can assume the following: for each edge of $\Phi$, say with endpoints $j$ and $j+1$, $b$ represents a braid which the $j^{th}$ and $(j+1)^{st}$ strands run straight down, parallel to each other, and no strands cross between them. That is, $[(b, \Phi)] = [(\id, \Phi)]$. Thus $[(b, \Phi \cup \Gamma)]$ is a simplex in $s\velbraigecpx_n$ with $[(b, \Gamma)]$ and $[(\id, \Phi)]$ as faces.
\end{proof}

In order to prove the connectivity of $s\velbraigecpx_{n}$ we first need to use Lemma \ref{lem:common_simplex_lemma} to prove the fiber of any $k$-simplex is $(k-1)$-connected. The proof of this proposition is similar to the monocolored case from \cite{Bux16}.

\begin{proposition}
	\label{prop:multicolored_fibers}
	Let $\sigma$ be a $k$-simplex in $s\matcharc(K_n)$ with vertices $v_0, \dots, v_k$. Then \[(s\pi)^{-1}(\sigma) = \Bigjoin_{j=0}^k(s\pi)^{-1}(v_j).\] In particular $(s\pi)^{-1}(\sigma)$ is $(k-1)$-connected.
\end{proposition}

\begin{proof}
	The complexes on either side have the same vertex sets so we need to show that a finite set of vertices spans a simplex on the left if and only if it spans a simplex on the right.
	
	``$\subseteq$'': This inclusion is equivalent to saying that vertices in the fiber $(s\pi)^{-1}(\sigma)$ that are connected by an edge map to distinct vertices under $s\pi$. This follows from the construction of $s\pi$.
	
	``$\supseteq$'': We will proceed by induction. The 0-skeleton of $\Bigjoin_{j=0}^k(s\pi)^{-1}(\sigma)$ consists of vertices which are mapped to vertices under $s\pi$, so it is contained in the fiber $(s\pi)^{-1}(\sigma)$. Now assume that the same is true of the $l$-skeleton, for some $l \geq 0$. Let $\tau$ be an $(l+1)$-simplex in $\Bigjoin_{j=0}^k(s\pi)^{-1}(v_j)$. Then decompose $\tau$ as the join of a vertex $[(b, \Gamma)]$ and an $l$-simplex $[(c, \Phi)]$. By induction, these are both in the fiber $(s\pi)^{-1}(\tau)$ and by Lemma \ref{lem:common_simplex_lemma} share a simplex in $s\velbraigecpx_n$.  The minimal dimensional such simplex maps into $\sigma$ under $s\pi$, hence we have our desired result. 
	
	The consequence of the resulting equality is that $(s\pi)^{-1}(\sigma)$ is the join of $k+1$ nonempty things, which is $(k-1)$-connected.
\end{proof}

We will now prove the high connectivity of $s\velbraigecpx_n$. Recall that for $n \in \Z$, $\nu(n) = \floor{\frac{n+1}{3}} - 1$.

\begin{corollary}
	\label{cor:sVEn_conn}
	$s\velbraigecpx_n$ is $(\nu(n) - 1)$-connected. 
\end{corollary}

\begin{proof}
	We wish to satisfy the conditions of \cite[Theorem 9.1]{quillen78}. We proved in Theorem \ref{thrm:multicolored_surface_matching_conn} that $s\matcharc(K_n)$ is $(\nu(n) - 1)$-connected. For any $k$-simplex $\sigma$ in $s\matcharc(K_n)$, we know $(s\pi)^{-1}(\sigma)$ is $(k-1)$-connected by Proposition \ref{prop:multicolored_fibers}. Moreover, we know the link $\lk(\sigma)$ is found by deleting $2k + 2$ arcs from $\sigma$, hence $2k + 2$ points from the set of punctures $P$. So there are $n - 2k - 2$ punctures left behind and $\lk(\sigma)$ is all the combinations of things we could do with the $n - 2k - 2$ punctures, hence $\lk(\sigma)$ is isomorphic to $s\matcharc(K_{n - 2k-2})$. Therefore we apply Theorem \ref{thrm:multicolored_surface_matching_conn} to conclude $\lk(\sigma)$ is $(\nu(n - 2k -2) - 1)$-connected. By the inequality 
	
	\begin{tabular}{lll}
	$\nu(n - 2k - 2)$ & $=$ & $\floor{\frac{(n - 2k - 2) + 1}{3}} -1$ \\
	& $\geq$ & $\floor{\frac{n + 1}{3}} + \floor{\frac{-2(k + 1)}{3}} - 1$ \\ 
	& $=$ & $\nu(n) + \floor{\frac{-2k - 1}{3}}$ \\
	& $\geq$ & $\nu(n) + \floor{\frac{-2k}{3}} + \floor{\frac{-1}{3}}$ \\
	& $\geq$ & $\nu(n) - k - 1$ \\
	\end{tabular}
	\\
	we have $\lk(\sigma)$ is the desired $(\nu(n) - k - 2)$-connected. Therefore we have satisfied the hypotheses of \cite[Theorem 9.1]{quillen78} and we can conclude that $s\velbraigecpx_n$ is $(\nu(n) - 1)$-connected.
\end{proof}

\subsection{Connectivity of $s\elbraigecpx_n$, hence connectivity of the descending links}

We will now build $s\elbraigecpx_n$ starting from $s\velbraigecpx_n$ using a Morse function on the elements of $s\elbraigecpx_n$ inspired by another Morse function in \cite{Fluch13} for $sV$. For the rest of this paper, given a braige $(b,F)$ we will refer to the trees $F$ as ``merges,'' and refer to ``very elementary merges'' to mean a merge with two leaves. Define $\mu_k$ to be the number of merges that have $k$ leaves. This forms a function $\mu \defeq (\mu_{2^s}, \mu_{2^s - 1}, \dots, \mu_{4}, \mu_{3})$ which we order lexicographically. Moreover, recall $f$ to denote the number of feet of the given $n$-braige as before. Finally, define $h \defeq (\mu, f)$ ordered lexicographically which satisfies being a Morse function as in \cite{Fluch13}. See Figure \ref{fig:morse_function_sVbraid} for an illustration using the Morse function described here.

\begin{figure}[t]
	\centering
		\begin{tikzpicture}[line width = 1.5pt, scale=0.8]
			\coordinate (a) at (0,0);
			\coordinate (b) at ($(a) + (1,0)$);
			\coordinate (c) at ($(a) + (2,0)$);
			\coordinate (d) at ($(a) + (3,0)$);
			\coordinate (e) at ($(a) + (4,0)$);
			\coordinate (f) at ($(a) + (5,0)$);
			\coordinate (g) at ($(a) + (6,0)$);
			\coordinate (h) at ($(a) + (7,0)$);
			\coordinate (i) at ($(a) + (8,0)$);
			\coordinate (j) at ($(a) + (0.5,-1)$);
			\coordinate (k) at ($(a) + (1,-2)$);
			\coordinate (l) at ($(a) + (3,-3)$);
			\coordinate (m) at ($(a) + (4.5,-1)$);
			\coordinate (n) at ($(a) + (5,-2)$);
			\coordinate (o) at ($(a) + (5.5,-3)$);
			\coordinate (p) at ($(a) + (1,-3)$);
			\coordinate (q) at ($(a) + (8, -3)$);
			
			\draw
			(d) -- (l)
			(k) -- (p)
			(i) -- (q);
			
			\draw[red]
			(a) -- (j) -- (b)
			(m) -- (n) -- (g)
			(n) -- (o) -- (h);
			
			\draw[blue]
			(j) -- (k) -- (c)
			(e) -- (m) -- (f);
			
			\filldraw[red]
			(a) circle (1pt)
			(b) circle (1pt)
			(g) circle (1pt)
			(h) circle (1pt)
			(j) circle (1pt)
			(n) circle (1pt)
			(o) circle (1pt);
			
			\filldraw[blue]
			(c) circle (1pt)
			(e) circle (1pt)
			(f) circle (1pt)
			(k) circle (1pt)
			(m) circle (1pt);
			
			\filldraw
			(d) circle (1pt)
			(i) circle (1pt)
			(l) circle (1pt)
			(p) circle (1pt)
			(q) circle (1pt);

		\end{tikzpicture}
	\caption{Given the element $z$ shown, $\mu(z) = (1, 1)$ and $f(z) = 4$.}
	\label{fig:morse_function_sVbraid}
\end{figure}
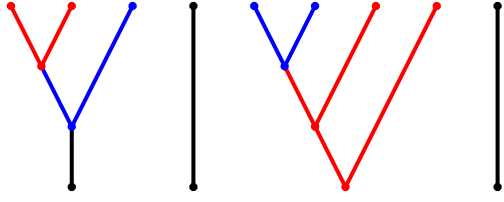

Fix $z$ to be an element of $s\elbraigecpx_n \backslash s\velbraigecpx_n$. We will analyze the descending link of $z$ that depends on this Morse function $h$ and denote it by $\lk_{\downarrow, h}(z)$. As we have a lexicographic order on $h$ there are two cases to consider: either $\mu$ decreases by adding legal splits to $z$ or $\mu$ stays the same and $f$ decreases by adding merges. We shall review these two cases for subcomplexes in this context. The full subcomplex spanned by vertices in the former case is called the \textit{(descending) split link} of $z$. The full subcomplex of vertices that satisfy the latter case form the \textit{(descending) merge link}. Consider two vertices $y$ and $w$ in the split link and merge link of $z$ respectively. By the definition of these subcomplexes, we already know $y$ is obtained from $z$ via splits and $w$ is obtained from $z$ via merges. See Figure \ref{fig:split_merge_links} for an illustration involving obtaining elements the split and merge links of a certain element.


\begin{figure}[t]
	\centering
		\begin{tikzpicture}[line width = 1.5pt]
			\begin{scope}[xshift=-3cm,yshift=1cm, scale=.6]
			\coordinate (a) at (0,0);
			\coordinate (b) at ($(a)+(1,0)$);
			\coordinate (c) at ($(a)+(2,0)$);
			\coordinate (d) at ($(a)+(3,0)$);
			\coordinate (e) at ($(a)+(4,0)$);
			\coordinate (f) at ($(a)+(5,0)$);
			\coordinate (g) at ($(a)+(0,-1)$);
			\coordinate (h) at ($(a)+(1,-1)$);
			\coordinate (i) at ($(a)+(2,-1)$);
			\coordinate (j) at ($(a)+(3,-1)$);
			\coordinate (k) at ($(a)+(0,-3)$);
			\coordinate (l) at ($(a)+(1,-3)$);
			\coordinate (m) at ($(a)+(2,-3)$);
			\coordinate (n) at ($(a)+(3,-3)$);
			\coordinate (o) at ($(a)+(0.5,-4)$);
			\coordinate (p) at ($(a)+(1,-5)$);
			\coordinate (q) at ($(a)+(1.5,-6)$);
			\coordinate (r) at ($(a)+(4,-6)$);
			\coordinate (s) at ($(a)+(5,-6)$);
			
			\draw
			 (a) -- (g)
			 (b) -- (h)
			 (c) -- (i)
			 (d) -- (n)
			 (e) -- (r)
			 (f) -- (s);
			 \draw[red]
			 (k) -- (o) -- (l)
			 (p) -- (q) -- (n);
			 \draw[blue]
			 (o) -- (p) -- (m);
			 
			 \draw
			 (g) to [out=-90, in=90] (l)
			 (h) to [out=-90, in=90] (m);
			 
			\draw[white, line width=3pt]
			(i) to [out=-90, in=90] (k);
			\draw
			(i) to [out=-90, in=90] (k);
			
			\filldraw
			(a) circle (1pt)
			(b) circle (1pt)
			(c) circle (1pt)
			(d) circle (1pt)
			(e) circle (1pt)
			(f) circle (1pt)
			(g) circle (1pt)
			(h) circle (1pt)
			(i) circle (1pt)
			(j) circle (1pt)
			(r) circle (1pt)
			(s) circle (1pt);
			\filldraw[red]
			(k) circle (1pt)
			(l) circle (1pt)
			(q) circle (1pt)
			(n) circle (1pt);
			\filldraw[blue]
			(m) circle (1pt)
			(o) circle (1pt)
			(p) circle (1pt);
		\end{scope}
		
		\node[rotate=25] at (1,0.25) {$\to$};
		\node at (1,0.75) {merge};
		\node[rotate=-25] at (1,-2) {$\to$};
		\node at (1,-2.5) {split};
	\begin{scope}[xshift=3cm, yshift=2.5cm, scale=.5]
		\coordinate (a) at (0,0);
		\coordinate (b) at ($(a)+(1,0)$);
		\coordinate (c) at ($(a)+(2,0)$);
		\coordinate (d) at ($(a)+(3,0)$);
		\coordinate (e) at ($(a)+(4,0)$);
		\coordinate (f) at ($(a)+(5,0)$);
		\coordinate (g) at ($(a)+(0,-1)$);
		\coordinate (h) at ($(a)+(1,-1)$);
		\coordinate (i) at ($(a)+(2,-1)$);
		\coordinate (j) at ($(a)+(3,-1)$);
		\coordinate (k) at ($(a)+(0,-3)$);
		\coordinate (l) at ($(a)+(1,-3)$);
		\coordinate (m) at ($(a)+(2,-3)$);
		\coordinate (n) at ($(a)+(3,-3)$);
		\coordinate (o) at ($(a)+(0.5,-4)$);
		\coordinate (p) at ($(a)+(1,-5)$);
		\coordinate (q) at ($(a)+(1.5,-6)$);
		\coordinate (r) at ($(a)+(4,-5)$);
		\coordinate (s) at ($(a)+(5,-5)$);
		\coordinate (t) at ($(a)+(4.5,-6)$);
		
		\draw
		(a) -- (g)
		(b) -- (h)
		(c) -- (i)
		(d) -- (n)
		(e) -- (r)
		(f) -- (s);
		\draw[red]
		(k) -- (o) -- (l)
		(p) -- (q) -- (n)
		(r) -- (t) -- (s);
		\draw[blue]
		(o) -- (p) -- (m);
		
		\draw
		(g) to [out=-90, in=90] (l)
		(h) to [out=-90, in=90] (m);
		
		\draw[white, line width=3pt]
		(i) to [out=-90, in=90] (k);
		\draw
		(i) to [out=-90, in=90] (k);
		
		\filldraw
		(a) circle (1pt)
		(b) circle (1pt)
		(c) circle (1pt)
		(d) circle (1pt)
		(e) circle (1pt)
		(f) circle (1pt)
		(g) circle (1pt)
		(h) circle (1pt)
		(i) circle (1pt)
		(j) circle (1pt);
		
		\filldraw[red]
		(k) circle (1pt)
		(l) circle (1pt)
		(q) circle (1pt)
		(n) circle (1pt)
		(r) circle (1pt)
		(s) circle (1pt)
		(t) circle (1pt);
		
		\filldraw[blue]
		(m) circle (1pt)
		(o) circle (1pt)
		(p) circle (1pt);
	\end{scope}

	\begin{scope}[xshift=3cm, yshift=-1.5cm, scale=.5]
		\coordinate (a) at (0,0);
		\coordinate (b) at ($(a)+(1,0)$);
		\coordinate (c) at ($(a)+(2,0)$);
		\coordinate (e) at ($(a)+(3,0)$);
		\coordinate (f) at ($(a)+(4,0)$);
		\coordinate (g) at ($(a)+(0,-1)$);
		\coordinate (h) at ($(a)+(1,-1)$);
		\coordinate (i) at ($(a)+(2,-1)$);
		\coordinate (k) at ($(a)+(0,-3)$);
		\coordinate (l) at ($(a)+(1,-3)$);
		\coordinate (m) at ($(a)+(2,-3)$);
		\coordinate (o) at ($(a)+(0.5,-4)$);
		\coordinate (p) at ($(a)+(1,-5)$);
		\coordinate (r) at ($(a)+(3,-5)$);
		\coordinate (s) at ($(a)+(4,-5)$);
		
		\coordinate (y) at ($(a)+(5,0)$);
		\coordinate (z) at ($(a)+(5,-5)$);
		
		\draw
		(a) -- (g)
		(b) -- (h)
		(c) -- (i)
		(e) -- (r)
		(f) -- (s)
		(y) -- (z);
		\draw[red]
		(k) -- (o) -- (l);
		\draw[blue]
		(o) -- (p) -- (m);
		
		\draw
		(g) to [out=-90, in=90] (l)
		(h) to [out=-90, in=90] (m);
		
		\draw[white, line width=3pt]
		(i) to [out=-90, in=90] (k);
		\draw
		(i) to [out=-90, in=90] (k);
		
		\filldraw
		(a) circle (1pt)
		(b) circle (1pt)
		(c) circle (1pt)
		(e) circle (1pt)
		(f) circle (1pt)
		(g) circle (1pt)
		(h) circle (1pt)
		(i) circle (1pt)
		(r) circle (1pt)
		(s) circle (1pt)
		(y) circle (1pt)
		(z) circle (1pt);
		\filldraw[red]
		(k) circle (1pt)
		(l) circle (1pt)
		(q) circle (1pt);
		\filldraw[blue]
		(m) circle (1pt)
		(o) circle (1pt)
		(p) circle (1pt);
		
	\end{scope}
	\end{tikzpicture}
	
	\caption{Here is an example of an element (left), an element of its merge-link (top-right), and an element of its split-link (bottom-right).}
	\label{fig:split_merge_links}
\end{figure}
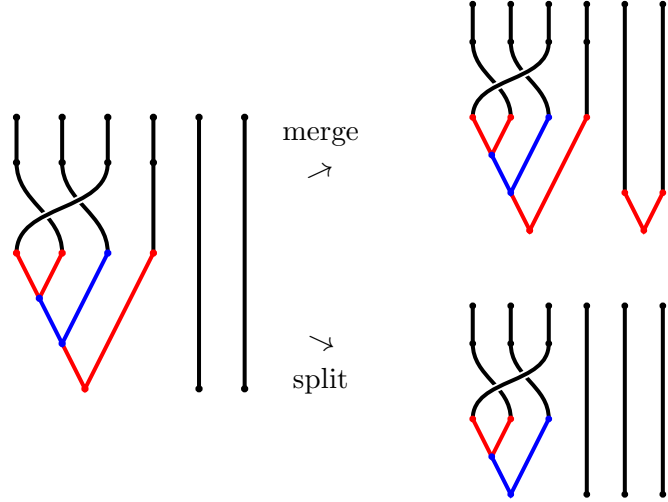

The following provides two propositions which show the importance of the two subcomplexes of $\lk_{\downarrow, h}(x)$ where $x \in s\elbraigecpx_n$.

\begin{proposition}
	Let $x$ and $y$ be vertices in $s\elbraigecpx_n$ with $x < y$. Then $\mu(x) \geq \mu(y)$ and $f(x) < f(y)$. Moreover, $h(x) < h(y)$ if and only if $\mu(x) = \mu(y)$, and $h(x) > h(y)$ if and only if $\mu(x) > \mu(y)$. 
\end{proposition}

\begin{proof}
	Observe that the inequality $x < y$ means that $y$ is obtained from $x$ via a nontrivial splitting, which implies the  ``feet function'' $f$ will have the desired inequality $f(x) < f(y)$. In addition, adding splits to $x$ is equivalent to $\mu_k$ changing in a decreasing way for some $k$. That is, if adding a split breaks up a merge with $k$ leaves then $\mu_k$ decreases, and if that is the first $\mu_k$ where $x$ and $y$ differ we get $\mu(x) > \mu(y)$. Otherwise we get an equality.
	
	Moreover, $h = (\mu, f)$ is ordered lexicographically. As we showed, $f(x) < f(y)$ and $\mu(x) \geq \mu(y)$. So $h(x) < h(y)$ if and only if  $\mu(x) = \mu(y)$.  Due to the lexicographic order and our previous observations, $h(x) > h(y)$ if and only if $\mu(x) > \mu(y)$.
\end{proof}

\begin{proposition}
	The descending link $\lk_{\downarrow, h}(z)$ is the join of the split link and merge link.
\end{proposition}

\begin{proof}
	Observe that every vertex of $\lk_{\downarrow, h}(z)$ is in either the split link or the merge link, but not both due to the previous proposition. Moreover, since any split of $z$ is also a split of any merge of $z$, then every simplex in the one subcomplex is joined to every simplex in the other subcomplex. 
\end{proof}

With this set up, we can now prove the connectivity of $s\elbraigecpx_n$. The following two lemmas will help do just that.

\begin{lemma}
	\label{lem:desc_morse_link_contra}
	When $z$ has a very elementary merge (so $\mu_2(z)$ is nonzero), we have the split link is contractible and so $\lk_{\downarrow, h}(z)$ is contractible.
\end{lemma}

\begin{proof}
	Let $y$ be a vertex of the split link and so in $s\elbraigecpx_n$. Since $y$ is in the split link, we know $y$ is obtained from $z$ via (possibly very elementary) splittings. Let $y_0$ be obtained from $z$ via doing all the same splittings as $y$ except not splitting any very elementary merges. Thus $y \geq y_0$ as $y$ can also be obtained from $y_0$ via splitting some very elementary merge(s). As $y_0$ is obtained by nontrivial splittings of $z$, we have $y_0 \geq z$. Moreover, we have $\mu(y_0) < \mu(z)$, thus $y_0$ is in the split link of $z$. Now let $z_0$ be the maximal element of the split link of $z$ obtained by not splitting any very elementary merges of $z$. Observe that $\mu_2(z)$ is nonzero, so $z_0$ is still in $s\elbraigecpx_n$. Due to $z_0$ being a maximal element, $y_0 \leq z_0$ for all $y$ in the split link. Thus $y \geq y_0 \leq z_0$. This provides a contraction of the split link of $z$. 
\end{proof}

Given a merge of $z$, only certain single-caret splittings can be applied to it to keep it a braige (not merely a spraige). For example, if $z$ has a merge consisting of a red caret and a blue caret under it, we can apply a single blue single-caret splitting, however, any other splits would no longer result in a braige. Call a merge a \emph{bottleneck merge} if it admits one color of single-caret splitting that yields a braige. 
 
\begin{lemma}
	\label{lem:bottleneck_merge}
	When $z$ has no very elementary merge (so $\mu_2(z) = 0$), but does have a bottleneck merge, then the split link is contractible, and so $\lk_{\downarrow, h}(z)$ is contractible.
\end{lemma}

\begin{proof}
	The proof of this lemma is similar to the previous one. Let $y$ be in the split link of $z$. Let $y_0$ be obtained from $y$ as follows: if $y$ involves splitting the bottleneck merge, then $y_0 = y$, and if $y$ does not involve splitting the bottleneck merge, then $y_0$ is $y$ except we additionally do the one legal single-caret splitting to the bottleneck. So far we have $z < y \leq y_0$ and $\mu(y_0) < \mu(z)$, as $y$ is obtained from $z$ via a nontrivial splitting and $y_0$ is obtained via $y$ through another (possibly trivial) splitting. So $y_0$ still lies in the split link of $z$. Now let $w$ be the element of the split link of $z$ obtained by just doing the one legal single-caret splitting to the bottleneck merge, and nothing else. Clearly $w \leq y_0$ due to the second case described previously, and due to the bottleneck property in fact we also have $w \leq y_0$ in the first scenario. Hence we have $w \leq y_0 \geq y$, which provides a contraction of the split link of $z$. 
\end{proof}

In the process of proving our main result, it was realized that there is a mistake in the proof of Lemma 3.8 of \cite{Fluch13}, namely the split link is not a join since it is missing the maximal simplex given by undoing every merge. That is, the split link in Lemma 3.8 of \cite{Fluch13} is the join of $k$ nonempty complexes with a simplex removed is not immediately $(k-1)$-connected, so we have the following lemma which proves this. In particular this fixes Lemma 3.8 of \cite{Fluch13}.

\begin{lemma}
	\label{lem:remove_simplex_connectivity}
	Let $S_1, \dots, S_k$ be 0-dimensional complexes. Let $X = S_1 \ast \cdots \ast S_k$, so $X$ is $(k-2)$-connected. Let $\Delta$ be a (maximal) $(k-1)$-simplex of $X$, say with vertices $s_i \in S_i$ for $1 \leq i \leq k$. Let $Y$ be $X$ with the interior of $\Delta$ removed. Assume that the inclusion $\partial \Delta \to Y$ is nullhomotopic. Then $Y$ is $(k-2)$-connected.
\end{lemma}

\begin{proof}	
	If $k = 1$ there is nothing to show since $Y$ is (-1)-connected (i.e., non-empty). If $k = 2$ then it is easy to see that $Y$ is 0-connected (i.e, connected). Now assume $k \geq 3$. In particular $\partial \Delta = Y \cap \Delta$ (and $Y$) are connected, and $X$ (and $\Delta$) are simply connected. Hence applying Van-Kampen to $X = Y \cup \Delta$ yields $\{1\} \cong \pi_1(Y) \ast_{\pi_1(\partial \Delta)}\{1\}$. Since the map $\pi_1(\partial \Delta) \to \pi_1(Y)$ is trivial, we conclude $\{1\} \cong \pi_1(Y)$, i.e, $Y$ is simply connected. Now thanks to Hurewicz it only remains to show that $Y$ is $(k-2)$-acyclic. For this we use Mayer–Vietoris. Since $\Delta$ is contractible, the Mayer–Vietoris sequence looks like $$\cdots \to H_i(\partial \Delta) \to H_i(Y) \to H_i(X) \to H_{i-1}(\partial \Delta) \to \cdots$$ for all $i > 1$. Since $X$ is $(k-2)$-connected it is also $(k-2)$-acyclic, so for all $i \leq k-2$ this becomes $$\cdots \to H_i(\partial \Delta) \to H_i(Y) \to 0.$$ We are assuming the inclusion $\partial \Delta \to Y$ is nullhomotopic, so it induces the trivial map in homology, and so by exactness we get that $H_i(Y) = 0$ for all $i \leq k-2$, as desired.
\end{proof}

\begin{lemma}
	\label{lem:boundary_inclusion_nullhomotopic}
	Under the set up of the previous lemma, if $|S_i| \geq 2$ for all $i$ then the inclusion $\partial \Delta \to Y$ is always nullhomotopic.
\end{lemma}

\begin{proof}
	For each $i$ choose $t_i \in S_i \setminus \{s_i\}$. Consider the subcomplex of $X$ given by $Z = \{s_1,t_1\} \ast \cdots \ast \{s_k,t_k\}$. Since $Z$ contains $\partial \Delta$, it is enough to show that $\partial \Delta \to Z \cap Y$ is nullhomotopic. Note that $Z$ is a triangulation of $S^{k-1}$, namely it is the $(k-1)$-dimensional hyperoctahedron. The subcomplex $\partial \Delta$ is a copy of $S^{k-2}$ in here. It separates $Z$ into two components, each of which is a $(k-1)$-disk: one of them is the interior of $\Delta$ and the other is $Z \setminus \Delta$. In particular $\partial \Delta$ bounds a $(k-1)$-disk in $Z \cap Y$, namely $Z \setminus \Delta$. We conclude that $\partial \Delta \to Z \cap Y$ and hence $\partial \Delta \to Y$ is nullhomotopic.
\end{proof}

For $k \in \Z$ and fixed $s > 1$, define $\eta(k) \defeq \floor{\frac{k - 2}{2^s}}$ and recall that $\nu(k) = \floor{\frac{k - 2}{3}}$. Observe that $\eta(k)$ grows monotonically towards $\infty$ as $k$ goes to  $\infty$.

\begin{figure}[t]
	\centering
	\begin{tikzpicture}[line width=1.5pt, scale=.75]
		\coordinate (a) at (0,0);
		\coordinate (b) at ($(a) + (0, -3)$);
		\coordinate (c) at ($(a) + (2, 0)$);
		\coordinate (d) at ($(a) + (3,0)$);
		\coordinate (e) at ($(a) + (4,0)$);
		\coordinate (f) at ($(a) + (5,0)$);
		\coordinate (g) at ($(a) + (2.5,-1)$);
		\coordinate (h) at ($(a) + (3,-2)$);
		\coordinate (i) at ($(a) + (3.5,-3)$);
		
		\draw
		(a) -- (b);
		\draw[red]
		(c) -- (g) -- (d)
		(g) -- (h) -- (e);
		\draw[blue]
		(h) -- (i) -- (f);
		
		\filldraw
		(a) circle (1pt)
		(b) circle (1pt);
		\filldraw[red]
		(c) circle (1pt)
		(d) circle (1pt)
		(e) circle (1pt)
		(g) circle (1pt)
		(h) circle (1pt);
		\filldraw[blue]
		(f) circle (1pt)
		(i) circle (1pt);

		(e) circle (1pt)
		
		(i) circle (1pt);
		
	\end{tikzpicture}
	\caption{On the left is an example of an unmerged head and on the right is an example of a non-very elementary merge.}
	\label{fig:unmerged_foot_large_foot}
\end{figure}
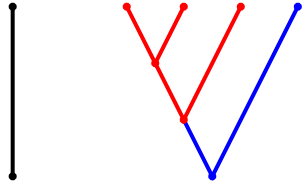

\begin{lemma}
	\label{lem:desc_morse_link_conn}
	When $z$ has no very elementary merges and no bottleneck merges, $\lk_{\downarrow, h}(z)$ is $(\eta(n) - 1)$-connected.
\end{lemma}

\begin{proof}
	First note that if we consider $s = 1$, then very elementary and elementary mean the same thing and there is nothing more to show. We will maintain that $s > 1$. Observe that there are no very elementary merges, so the following mentions of merges are elementary merges. That is, every collection of merges will consist of at least 3 strands. Call a strand not involved in any merges an \textit{unmerged head}. Also recall that a non-very elementary merge is a merge with at least three feet. See Figure \ref{fig:unmerged_foot_large_foot} for an example of each of these. Let $k_u$ be the number of unmerged heads and $k_e$ be the number of non-very elementary merges. Observe that each non-very elementary merge can have at most one merge of each color in order to stay elementary, thus each non-very elementary merge has at most $2^s$ feet. Hence the number of non-very elementary merges satisfies $\frac{n - k_u}{2^s} \leq k_e$. Therefore by multiplication of $2^s$ we have $n - k_u \leq 2^sk_e$, hence $n \leq 2^sk_e + k_u$.
	
	Since $k_e$ counts the number of non-very elementary merges there are, then there are $k_e$ ways of splitting non-very elementary merges. All but one of these merges could be split and results in an element of $s\elbraigecpx_n$. The one scenario where we are not allowed to split a non-very elementary merge is when we split all of the merges at once as the resulting vertex would no longer be in $s\elbraigecpx_n$. So all legal ways of splitting non-very elementary merges would reduce $\mu_k(z)$ for some $k$, resulting in a vertex of lower height. Thus all but the largest merge contribute a nonempty vertex to the split link. This shows that the split link is the join of all these vertices with a simplex removed representing the case where all merges are split. Observe that the join of $k_e$ nonempty things is $(k_e - 2)$-connected. Moreover, the simplex being removed that represents the maximal number of splits we can apply to $z$ is itself a $k_e$-simplex. Since $z$ has no bottlenecks, by Lemma \ref{lem:boundary_inclusion_nullhomotopic} we have the boundary of this complex included into the simplex with the interior of the simplex removed is nullhomotopic. Therefore, by Lemma \ref{lem:remove_simplex_connectivity} we have the whole complex with the maximal simplex removed is $(k_e - 2)$-connected. This complex is the split link, thus the split link is $(k_e - 2)$-connected.
	
	Let $X \subseteq \{1, 2, \dots, n\}$ be the set of indices of the strands that were already involved in a non-very elementary merge of $z$. Denote the subcomplex of $s\velbraigecpx_{n}$ which does not allow strands indexed by elements of $X$ to be merged by $s\velbraigecpx_{n, X}$. Since the indices in $X$ are not allowed to be used in further merges, it follows that each element in the merge link that does not use a strand labeled by $X$, i.e. every combination of merges applied to $z$ with indices not from $X$, represents an element of $s\velbraigecpx_{n, X}$. Hence the merge link of $z$ is isomorphic to $s\velbraigecpx_{n, X}$. Observe $s\velbraigecpx_{n, X}$ maps to $s\matcharc(K_{n - |X|}, S;P)$ where $S$ is a surface with $X$ as its boundary points by restricting $(s\pi)$ to $s\velbraigecpx_{n, X}$ and the colored arcs in the codomain do not have endpoints in $X$. By Corollary \ref{cor:sVEn_conn} where $n$ is replaced with $n - |X|$, $s\velbraigecpx_{n, X}$ and hence the merge link is $(\nu(n - |X|) - 1)$-connected. 
	
	As $\lk_{\downarrow, h}(z)$ is the join of the split link and the merge link and $k_e - 2 + \nu(n - |X|) - 1 + 2 = k_e + \nu(n - |X|) - 1$, we have that $\lk_{\downarrow, h}(z)$ is $(k_e + \nu(n - |X|) - 1)$-connected.  Recall that $s > 1$ so $2^s > 3$. Due to properties of inequalities and $n - |X| = k_u$ we have: \\
	
	\centering
	\begin{tabular}{lll}
		$k_e + \nu(n - |X|) - 1$ &  $\geq $ & $\frac{n - k_u}{2^s} + \nu(n - |X|) - 1 $\\
		& $=$ & $\frac{n - k_u}{2^s} + \floor{\frac{n - |X| - 2}{3}} - 1 $ \\
		& $\geq$ & $\frac{n - k_u}{2^s} + \frac{n - |X| - 2}{3} - 2 $ \\
		& $\geq$ & $\frac{n - k_u}{2^s} + \frac{n - |X| - 2}{2^s} - 2 $ \\
		& $=$ & $\frac{n - k_u}{2^s} + \frac{k_u - 2}{2^s} - 2 $ \\
		& $=$ & $\frac{n - 2}{2^s} - 2$ \\
		& $\geq$ & $\floor{\frac{n - 2}{2^s}} - 2$ \\
		& $=$ & $\eta(n) - 2$\\
	\end{tabular}	
\\
By this inequality, we can conclude that $\lk_{\downarrow,h}(z)$ is at least $(\eta(n) - 2)$-connected.
\end{proof}

We are now able to prove the desired result of this subsection, allowing us to prove the main result of this paper. The Morse Lemma as stated in Section \ref{sec:def_stein_space} will be of importance for the following two proofs.

\begin{corollary}
	\label{cor:multicolored_desc_link_conn}
	$s\elbraigecpx_n$ is $(\eta(n)-1)$-connected. Hence for any vertex $x$ in $sX_{br}$ with $f(x) = n$, $\lk_{\downarrow}(x)$ $(\eta(n)-1)$-connected.
\end{corollary}

\begin{proof}
	If $s = 1$, by \cite[Theorem 3.8]{Bux16} $s\elbraigecpx_n = s\velbraigecpx_n$ is $(\eta(n) - 1)$-connected. Now assume $s > 1$. By Corollary \ref{cor:sVEn_conn} and that $\eta \leq \nu$ we know $s\velbraigecpx_{n}$ is at least $(\eta(n) - 1)$-connected. Moreover, given $z \in s\elbraigecpx_n \setminus s\velbraigecpx_{n}$, we know by Lemma \ref{lem:desc_morse_link_contra}, Lemma \ref{lem:bottleneck_merge}, and Lemma \ref{lem:desc_morse_link_conn}, $\lk_{\downarrow, h}(z)$ is contractible or at least $(\eta(n) - 2)$-connected. It follows from the Morse Lemma that $s\elbraigecpx_n$ is at least $(\eta(n) - 1)$-connected.
\end{proof}

Along with the previous corollary we also need the connectivity of the pair $(sX_{br}^{\leq n}, sX_{br}^{<n})$ for $n \geq 1$.

\begin{proposition}
	\label{prop:multicolored_filtration_conn}
	For each $n \geq 1$, the pair $(sX_{br}^{\leq n}, sX_{br}^{<n})$ is $(\eta(n) - 1)$-connected for $s > 1$, and $(1X_{br}^{\leq n}, 1X_{br}^{<n})$ is $(\nu(n) - 1)$-connected. 
\end{proposition}

\begin{proof}
	Let $x$ be a vertex in $sX_{br}^{=n}$. By Corollary \ref{cor:multicolored_desc_link_conn}, $\lk_{\downarrow}(x)$ is at least $(\eta(n) - 1)$-connected for $s > 1$, or $(\nu(n) - 1)$-connected when $s=1$. Therefore, by the Morse Lemma, we have the pair $(sX_{br}^{\leq n}, sX_{br}^{<n})$ is $(\eta(n) - 1)$-connected for $s > 1$, and $(1X_{br}^{\leq n}, 1X_{br}^{<n})$ is $(\nu(n) - 1)$-connected.
\end{proof}

\section{Proof of the main theorem}
\label{sec:proof_main_theorem}
We are now prepared to apply Brown's Criterion.

\begin{proof}[Proof of Main Theorem]	
	We wish to apply Brown's Criterion to $s\Vbr$. By Corollary \ref{cor:sX_br_contractible} we know that $sX_{br}$ is contractible. Moreover, we know the stabilizer of every simplex is of type $\F_\infty$ by Lemma \ref{lem:sX_br_cell_stab}. We also know that the filtration  $(sX_{br}^{\leq n})_{n \in \N}$ is cocompact by Lemma \ref{lem:sX_{br}_cocompact} and lastly that the pair $(sX_{br}^{\leq n}, sX_{br}^{<n})$ is $(\eta(n) - 1)$-connected or if $s = 1$ then $(1X_{br}^{\leq n}, 1X_{br}^{<n})$ is $(\nu(n) - 1)$-connected. Therefore by Brown's Criterion we conclude that $s\Vbr$ is of type $\F_\infty$.
\end{proof}



\newcommand{\etalchar}[1]{$^{#1}$}

%
%
%
%

%
%
%
%
%
%
%
%

\end{document}